\documentclass[12pt]{article}
\usepackage{amsgen, mathtools, amsmath, amsfonts, amsthm, amssymb, enumerate,amscd,graphicx}
\usepackage{tikz-cd}
\usepackage[colorlinks=true,linkcolor=black,anchorcolor=black,citecolor=black,filecolor=black,menucolor=black,runcolor=black,urlcolor=black]{hyperref}

\newcommand{\LL}{\mathbb{L}}
\newcommand{\GG}{\mathbb{G}}
\newcommand{\EE}{\mathcal{E}}

\newcommand{\keywords}[1]{\textbf{{Key Words:}} #1}

\numberwithin{equation}{section}
\newtheorem{theorem}{Theorem}[section]
\newtheorem{lemma}[theorem]{Lemma}

\theoremstyle{definition}  
\newtheorem{definition} [theorem] {Definition}

 \newcommand {\reg}{{\overrightarrow{1}_{\infty}}}

\begin{document}

\title{Laplace-eigenvalue equations for length three modular iterated integrals}
\date{}
\author
{Joshua Drewitt \\ \small University of Nottingham, University Park,
Nottingham,
United Kingdom}

\maketitle

\begin{abstract}
    A space of modular iterated integrals sits inside the space of real analytic modular forms. We present a theorem for producing length three modular iterated integrals which are not simply combinations of real analytic Eisenstein series; each function has an associated Laplace-eigenvalue equation. This can be viewed as an extension of the length two case recently given by F. Brown, a review of which is included in this paper. We discuss how modular iterated integrals could help understand the modular graph functions which arise in string perturbation theory.
\end{abstract}


\section{Introduction}
\subsection{Overview of Modular Iterated Integrals}

The space of real analytic modular forms was recently introduced by Brown \cite{brown1,brown2,brown3} and contains or intersects various previously studied classes of important modular objects, including classical modular forms, weakly anti-holomorphic forms and Maass wave forms. A newly identified space of modular iterated integrals, denoted by $\mathcal{MI}$, also sits naturally inside this class of forms. 

This space is extremely important for many reasons. It is closely related to the theory of single polylogarithms and mixed motives. Furthermore, in \cite{nikome} it is shown that we can associate period polynomials to certain modular iterated integrals; this is analogous to the period polynomials of classical modular forms. 

We will focus, however, on its connection to the modular graph functions which arise in string perturbation theory. These functions share many of the differential and algebraic properties of modular iterated integrals and, indeed, there is good evidence that all modular graph functions are contained within the space $\mathcal{MI}$ (this is discussed in more detail in \cite{brown1,brown2}). Therefore, studying the space of modular iterated integrals will help with the long standing problem of giving a complete description of the modular graph functions. This problem has been written about extensively in physics literature, see any of \cite{Proof123loop,modstructure,Idents,Hier,holosub,toclosed,gener,allorder,ellipzeta,heterotic}
 for example. 

The space $\mathcal{MI}$ is given an increasing length filtration 
\begin{equation*}
\mathcal{MI}_n \subset \mathcal{MI}, \qquad n \geq 0,
\end{equation*}
with an exact definition given in Section \ref{mii}.
The modular iterated integrals of length zero and one are well understood. We are able to give a concise description of both spaces:
\begin{equation*}
        \mathcal{MI}_0= \mathbb{C}[\mathbb{L}^{-1}] \hspace{14mm} \text{and} \hspace{14mm} \mathcal{MI}_1 = \, \mathcal{MI}_0 \otimes_{\mathbb{C}} \! \! \!    \bigoplus_{r,s \geq 0,\, r+s \geq 2}  \! \! \! \mathbb{C} \,\mathcal{E}_{r,s},
\end{equation*}
where $\mathbb{L}= i \pi (z-\bar{z})$, $z$ is an element of the upper half plane and $\mathcal{E}_{r,s}$ denotes the real analytic Eisenstein series of weights $(r,s)$ (Definition \ref{ers}).
 Therefore, we know of all the possible functions that can exist in $\mathcal{MI}_0$ or $\mathcal{MI}_1$. The same cannot be said, however, for higher lengths.
 
Brown \cite{brown1}, nonetheless, has found numerous functions that exist in $\mathcal{MI}_2$, explicit examples include 
\begin{equation*}
    \mathcal{E}_{0,2}\mathcal{E}_{1,1}, \hspace{9mm} \mathcal{E}_{2,0}\mathcal{E}_{0,2}+ \frac{1}{2}\mathcal{E}_{1,1}\mathcal{E}_{1,1} \hspace{9mm }\text{and} \hspace{9mm} \frac{1}{2}\mathcal{E}_{2,0}\mathcal{E}_{2,0},
\end{equation*}
however these are just linear combinations of real analytic Eisenstein series and we do not consider them ``new" functions. By ``new" functions we will therefore mean  functions that are not simply combinations of $\mathcal{E}_{r,s}$ or $\mathbb{L}$.

A theorem for producing such functions was discovered by Brown in Section 9 of \cite{brown1}. Each of these functions has an associated Laplace-eigenvalue equation, for example, the functions $F^{(1)}_{2,0}$, $F^{(1)}_{1,1}$ and $F^{(1)}_{0,2}$ satisfy
\begin{align}
    (\Delta+2)F^{(1)}_{2,0} &= -4\mathbb{L}^2\mathbb{G}_4\mathcal{E}_{0,2}, \label{intro1a} \\
    (\Delta+2)F^{(1)}_{1,1} &= -4 \mathbb{L}^3 \mathbb{G}_4  \overline{\mathbb{G}}_4, \label{intro1b} \\
    (\Delta+2)F^{(1)}_{0,2} &= -4\mathbb{L}^2\overline{\mathbb{G}}_4\mathcal{E}_{2,0},  \label{intro1c}
\end{align}
where $\mathbb{G}_4$ is an Eisenstein series \eqref{Eisen} and $\Delta$ is the Laplacian \eqref{lap}. We will see the importance of these Laplace equations for the ``new" functions in relation to modular graph functions shortly. 

\subsection{Modular Graph Functions}\label{sectmgf}

Modular graph functions arise in the low-energy expansion of Type II superstring amplitudes at genus one by assigning a lattice sum to a graph. Simply put, they are a type of real analytic function on the upper half plane. These functions have an extremely interesting mathematical structure, provide connections with number theory and particle phenomenology and are of great importance in other areas of physics. They have been discussed in detail in physics literature and we refer the reader to any of \cite{ModGraphFunc,Proof123loop,Lowengexp,singleval,modstructure} for more information on their origin.

Throughout this paper, we will use $z$ to denote an element of the upper half plane. The type of modular graph functions of interest to us are defined by the multiple sum
\begin{equation*}
    C_{a_1,\dots,a_p}(z) \! = \sum_{(m_i,n_i)  \neq (0,0)} \!\! \delta(m)\delta(n) \prod^p_{i=1}\left( \dfrac{Im(z)}{\pi|m_iz+n_i|^2}\right)^{a_i},
\end{equation*}
where $i=1,2,\dots,p$, each $m_i, n_i \in \mathbb{Z}$, $m=m_1+m_2+\cdots +m_p$, $n=n_1+n_2+\cdots+ n_p$, $\delta(x)=1$ if and only if $x=0$ and each $a_i$ is an integer greater than zero. Such a function is said to have $p-1$ loops and a total weight of $a_1+\cdots+a_p$. 
The space of one-loop modular graph functions is well understood. For each total weight there is a unique modular graph function that can be represented in terms of a real analytic Eisenstein series and $\mathbb{L}$:
\begin{equation*}
    C_{a_1,a_2} = \dfrac{2^{a+1}}{(2a-2)!}\mathbb{L}^{a-1}\mathcal{E}_{a-1,a-1},
\end{equation*}
where $a=a_1+a_2$. 
These one-loop modular graph functions obey the following Laplace-eigenvalue equation: 
\begin{equation*}
    \Delta C_{a_1,a_2} = a(a-1)C_{a_1,a_2}.
\end{equation*}
For example, at weight 4 we have
\begin{equation*}
   C_{2,2} = C_{1,3}= C_{3,1}  := \mathbb{C}_4  = \dfrac{2}{45} \mathbb{L}^3\mathcal{E}_{3,3} \qquad \
\text{and} \qquad \
     \Delta \mathbb{C}_4   = 12 \mathbb{C}_4. 
\end{equation*}

The space of two-loop modular graph functions is also well understood. The authors of \cite{modstructure}, building upon the work of \cite{Lowengexp} and \cite{theloweng}, discovered that these functions obey systems of inhomogeneous Laplace-eigenvalue equations. The inhomogeneous parts of these equations are made up of linear and quadratic combinations of real analytic Eisenstein series (and $\mathbb{L}$). Some examples from \cite{modstructure} include 
\begin{align}
    \Delta C_{2,2,1} &= \frac{4}{315}\mathbb{L}^4\mathcal{E}_{4,4}, \label{intro2a} \\
    (\Delta+2)C_{2,1,1} &= 16\mathbb{L}^2\mathcal{E}_{1,1}\mathcal{E}_{1,1}- \dfrac{2}{5}\mathbb{L}^3\mathcal{E}_{3,3},  \\
    (\Delta+6)C_{3,1,1} &= \frac{32}{3}\mathbb{L}^3\mathcal{E}_{1,1}\mathcal{E}_{2,2} - \frac{8}{315}\mathbb{L}^4\mathcal{E}_{4,4} -3C_{2,2,1} . \label{intro2b}
\end{align}
This first equation was solved easily in \cite{modstructure}, giving an explicit expression for $C_{2,2,1}$ in terms of a real analytic Eisenstein series (and a constant): 
\begin{equation}\label{c221}
   C_{2,2,1} = \frac{1}{1575}\mathbb{L}^4\mathcal{E}_{4,4}+ \frac{\zeta (5)}{30}.   
\end{equation}

In \cite{ellipzeta}, it was shown that $C_{2,1,1}$ and $C_{3,1,1}$ (or closely related objects) can be expressed in terms of convergent versions of iterated Eisenstein series. However, in this paper, we will mainly be interested in representations involving ``new'' functions, such as $F^{(1)}_{2,0}$, $F^{(1)}_{1,1}$ and $F^{(1)}_{0,2}$.
In \cite{brown1}, it was shown that we should be able to give an explicit expression for $C_{2,1,1}$ in terms of $\mathbb{L}F^{(1)}_{1,1}, \mathbb{L}^2\mathcal{E}_{2,0}\mathcal{E}_{0,2}$, $\mathbb{L}^3\mathcal{E}_{3,3}$ and a constant, since
\begin{equation}\label{match}
    (\Delta+2)\Big(4\mathbb{L}F^{(1)}_{1,1} -16\mathbb{L}^2\mathcal{E}_{2,0}\mathcal{E}_{0,2}+\dfrac{1}{25}\mathbb{L}^3\mathcal{E}_{3,3}\Big)= (\Delta+2)C_{2,1,1}.
\end{equation} 
This is discussed in more detail in Section \ref{trip}.
If other modular graph functions, such as $C_{3,1,1}$, can be expressed using ``new'' modular iterated integrals, then these integrals can help with the previously mentioned problem of giving a complete description of the modular graph functions. 

Similar to increasing the length of modular iterated integrals, as we increase the number of loops for the modular graph functions, the space of such functions becomes more obscure. Modular graph functions of three-loops (and higher) are no longer guaranteed to satisfy the linear or quadratic Laplace equations from the one- and two-loop cases. A few relations for low weight three- and four-loop functions have been conjectured and the simplest of these was recently proven in \cite{Proof123loop}.
See also \cite{Cabcd}, where the author obtains a Laplace-eigenvalue equation satisfied by three-loop modular graph functions $C_{a,b,c,d}$, for general $a,b,c$ and $d$.
However, a better understanding of these spaces is needed. 

\subsection{Overview of Results}

Whilst there is an extensive amount of literature in physics focusing on modular graph functions, the number theoretic literature on modular iterated integrals is limited. With this in mind, we aim to produce analogous examples of Brown's Laplace-eigenvalue equations but for ``new" functions in $\mathcal{MI}_3$ instead of $\mathcal{MI}_2$.  We will provide a simplified explanation of how this was achieved, a full explanation is given in Section \ref{length3}.

We let $a,b \geq 1$ and $2a+2b \leq 8$. By Theorem 9.3 of \cite{brown1}, we know there exists a family of elements $(F^{(0)}_{2a+2,2b+2})_{r,s} \in \mathcal{MI}_2$ of weights $(r,s)$, with $r,s \geq 0$ and $r+s=2a+2b$, which satisfy the following:
\begin{align*}
   &   \partial (F^{(0)}_{2a+2,2b+2})_{2a+2b,0} =  \mathbb{L}\mathbb{G}_{2a+2}\mathcal{E}_{2b,0} ,\\
& \partial (F^{(0)}_{2a+2,2b+2})_{r,s} - (r+1)(F^{(0)}_{2a+2,2b+2})_{r+1,s-1} = \mathbb{L}\mathbb{G}_{2a+2}\mathcal{E}_{2b-s,s}   ,  \tag*{if $ s \geq 1$.} \nonumber
\end{align*}
Here, to emphasise the importance of the values of $a$ and $b$, we use the convention $(F^{(0)}_{2a+2,2b+2})_{r,s}$ rather than the $F^{(0)}_{r,s}$ used by Brown. Using the above equations, we can construct the modular equivariant function
\begin{equation*}
    F^{(0)}_{2a+2,2b+2}(z) = \sum_{\mathclap{\substack{r+s= \\
    2a+2b}}} \, \big(F^{(0)}_{2a+2,2b+2}(z)\big)_{r,s}(X-zY)^r(X-\bar{z}Y)^s
\end{equation*}
which satisfies the differential equation
\begin{equation}\label{intro1}
   \dfrac{\partial }{\partial z} F^{(0)}_{2a+2,2b+2}(z)= \dfrac{1}{2} E_{2a+2}(z) \mathcal{E}_{2b}(z),
\end{equation}
where
\begin{align*}
    & {E}_{2a+2}(z) = 2 \pi i \, \mathbb{G}_{2a+2}(z)(X-z Y)^{2a},
     \\
    & \mathcal{E}_{2b} (z) = \sum_{r+s=2b} \mathcal{E}_{r,s}(X-z Y)^r(X-\bar{z}Y)^s.\label{e2a}
\end{align*}
Similarly, we can also show that the equation
\begin{equation}\label{intro2}
    \dfrac{\partial }{\partial \bar{z}} F^{(0)}_{2a+2,2b+2}(z) = \dfrac{1}{2}\mathcal{E}_{2a}(z)  \overline{{E}}_{2b+2}(z)
\end{equation}
holds. Now, letting $c\geq1$ and setting $2a+2b+2c \leq 8$, we then define the forms
\begin{equation*}
     D_{2a+2,2b+2,2c+2}(z):= E_{2a+2}(z) \cdot 
    F^{(0)}_{2b+2,2c+2}(z) \ dz + F^{(0)}_{2a+2,2b+2}(z) \cdot \overline{{E}}_{2c+2}(z) 
    \  d\bar{z}
\end{equation*}
and, using equations \eqref{intro1} and \eqref{intro2}, we can show this is a closed form. This leads to the definition of the function
\begin{equation*}
    K_{2a+2,2b+2,2c+2}(z) = -\dfrac{1}{2} \int_{z}^{\reg} D_{2a+2,2b+2,2c+2}.
\end{equation*}
Here the integral is regularised as in Section 8 of \cite{brown1} (see also, Section 4 of \cite{BrownMultiple}) and $\reg$ denotes the tangential base point at the cusp. An explanation of this regularisation is given in Section \ref{doub}. This function obeys the differential equation
\begin{align*}
     \dfrac{\partial}{\partial z} K_{2a+2,2b+2,2c+2}&(z) = \pi i  \mathbb{G}_{2a+2}(z)(X-zY)^{2a} \cdot F^{(0)}_{2b+2,2c+2}(z) 
     \\[1mm]
     =  \pi i &\mathbb{G}_{2a+2}(z)  \, \sum_{\mathclap{\substack{r+s= \\
    2a+2b+2c}}} \ \big(F^{(0)}_{2b+2,2c+2}(z)\big)_{r-2a,s}(X-zY)^r(X-\bar{z}Y)^s.
\end{align*}

This, in turn, leads to a simplified version of our main theorem of this paper (the full version of this theorem is Theorem \ref{3MI3THM}).
\begin{theorem}
Let $a,b,c \geq 1$, $2a+2b+2c\leq8$ and set $w=a+b+c$. There exists a family of elements $(G_{2a+2,2b+2,2c+2})_{r,s} \in \mathcal{MI}_3$ of weights $(r,s)$, with $r,s \geq 0$ and $r+s=2w$, which satisfy the following:
\begin{align*}
      & \partial (G_{2a+2,2b+2,2c+2})_{2w,0} =  \mathbb{L}\mathbb{G}_{2a+2} \big(F^{(0)}_{2b+2,2c+2}\big)_{2b+2c,0},\\
     & \partial (G_{2a+2,2b+2,2c+2})_{r,s} - (r+1)(G_{2a+2,2b+2,2c+2})_{r+1,s-1} \\ & \hspace{60mm}= \mathbb{L}\mathbb{G}_{2a+2} \big(F^{(0)}_{2b+2,2c+2}\big)_{r-2a,s}   ,  \tag*{if $s \geq 1$.}
\end{align*}
\end{theorem}

This theorem can produce the Laplace-eigenvalue equations for ``new'' length three modular iterated integrals we are looking for. Examples include
\begin{align}
    (\Delta+8)(G_{6,4,4})_{6,2}&=-\mathbb{L}^2\mathbb{G}_6\overline{\mathbb{G}}_4\mathcal{E}_{2,0} -3\mathbb{L}\mathbb{G}_6\mathcal{E}_{1,1}\mathcal{E}_{0,2}, \label{intro3a}\\
    (\Delta+8)(G_{6,4,4})_{7,1}&=-\mathbb{L}\mathbb{G}_6\mathcal{E}_{1,1}\mathcal{E}_{1,1} -2\mathbb{L}\mathbb{G}_6\mathcal{E}_{2,0}\mathcal{E}_{0,2}. \label{intro3c}
\end{align}

\subsection{Further Research}\label{research}

We can use Theorem \ref{3MI3THM} with Theorem 9.3 of \cite{brown1} to create sets of Laplace-eigenvalue equations in the length two and three cases. We can compare this to the data on modular graph functions from literature in physics to repeat the result of \eqref{match} but with different modular graph functions. This is a relatively new idea and so, as far as the author is aware, it has not yet been achieved on a larger scale than the one equation given by \eqref{match} and the examples we give in Section \ref{examples}. 

We recall that the two-loop modular graph functions $C_{a,b,c}$ satisfy systems of Laplace-eigenvalue equations whose inhomogeneous parts are linear/quadratic in real analytic Eisenstein series, as in equations \eqref{intro2a} -- \eqref{intro2b}. Similarly, the Laplace equations for length two modular iterated integrals can be seen as quadratic in some form of Eisenstein series ($\mathcal{E}_{r,s}$ or $\mathbb{G}_n$), as in \eqref{intro1a} -- \eqref{intro1c}. However, the Laplace equations in the length three case are cubic in some form of Eisenstein series, as in \eqref{intro3a} and \eqref{intro3c}. 

Therefore, it would not be surprising if the length three Laplace equations were not used in expressing two-loop modular graph functions. Instead, it is more likely that, if they were to be used,  it would be with expressing the much less understood three-loop (or higher) modular graph functions. See the end of Section \ref{endex3}, where we discuss this idea in more detail. It will also be interesting to see if any modular iterated integrals do not arise in the expression of any modular graph functions, or vice versa.

Of course, producing ``new'' length three modular iterated integrals will provide crucial insight into the space $\mathcal{MI}_3$, and $\mathcal{MI}$ in general, which is important by itself. Furthermore, this paper  provides a stepping stone to investigating ``new'' functions and Laplace-eigenvalue equations for higher length modular iterated integrals.

\vspace{1mm}

\textbf{Acknowledgements.} I thank N. Diamantis for his valuable support, feedback and encouragement throughout this writing process and F. Brown for his many helpful comments on modular iterated integrals. 
I am deeply grateful to O. Schlotterer for the many, many useful discussions related to this paper. I also thank A. Kleinschmidt, N. Matthes, M. Doroudiani, M. Hidding and B. Verbeek.
I am very thankful to the referees of this paper, who have greatly improved its content with their helpful suggestions and comments.
This work was supported by the Engineering and Physical Sciences Research Council (Grant No. EP/N50970X/1).

\section{Notation and Definitions}

We begin by setting some basic notation and recalling the definition of a real analytic modular form.

We let $z$ denote an element of the upper half plane $\mathcal{H}$ and set $\Gamma_1=SL(2,\mathbb{Z})$. For integers $r,s$, a function $f:\mathcal{H} \to \mathbb{C}$ and $\gamma =\begin{psmallmatrix}
a &  b \\
c &  d 
\end{psmallmatrix}$, we define the double slash operator by
\begin{equation*}
        (f ||_{r,s}\gamma)(z) = (cz+d)^{-r}(c\bar{z}+d)^{-s}f(\gamma z), \qquad  \forall \ z \in \mathcal{H}.
\end{equation*}
We extend this action to $\mathbb{C}[\Gamma_1]$ by linearity. For $z=x+iy$, we set
\begin{equation*}
    \mathbb{L} := i \pi (z - \bar{z})=-2 \pi y.
\end{equation*}

\begin{definition}
Let $r$ and $s$ be positive integers such that $r+s \in 2\mathbb{Z}$. A real analytic function $f: \mathcal{H} \to \mathbb{C}$ is a real analytic modular form of weights $(r,s)$ with respect to $\Gamma_1$ if
\begin{align}
(i)& \quad  (f||_{r,s} \gamma)(z) = f(z), \quad \forall \ \gamma \in \Gamma_1 \text{ and } z \in \mathcal{H}. \label{modinv2} \\
(ii)& \quad  \text{$f(z)$ has the Fourier expansion of the form} \nonumber \\
& \hspace{22mm} f(z)=\sum_{|j| \le M}y^j\left ( \sum_{m, n \ge 0} a_{m, n}^{(j)} \, e^{2 \pi i m z} e^{-2 \pi i n \bar{z}}\right ) \label{f=aqq} \\
  & \quad \text{ where $z=x+iy$, $M \in \mathbb{N}$ and $a_{m, n}^{(j)} \in \mathbb{C}$}. \nonumber
 \end{align}
\end{definition}

We denote the space of all modular forms of weights $(r,s)$ with respect to $\Gamma_1$ by $\mathcal{M}_{r,s}$ and set $\mathcal{M} = \bigoplus_{r,s}\mathcal{M}_{r,s}$. This class of functions was first introduced and studied by Brown in \cite{brown1, brown2, brown3}. 

The space $\mathcal{M}$ is equipped with the operators
\begin{equation*}
\partial_r: \mathcal{M}_{r,s} \to \mathcal{M}_{r+1,s-1} \quad \text{and} \quad \overline{\partial}_s: \mathcal{M}_{r,s} \to \mathcal{M}_{r-1,s+1}
\end{equation*}
given by
\begin{equation*}
\partial_r = 2iy \dfrac{\partial}{\partial z} +r \quad \text{and} \quad \overline{\partial}_s = -2iy \dfrac{\partial}{\partial \bar{z}} +s.
\end{equation*}
They induce bigraded operators $\partial,  \overline{\partial}: \mathcal{M} \to \mathcal{M}$ of bidegree $(1,-1)$ and $(-1,1)$, respectively. If $f \in \mathcal{M}_{r,s}$, then $\partial$ acts on $f$ via $\partial_r$ and similarly $\overline{\partial}$ acts via $\overline{\partial}_s$.

The Laplacian $\Delta_{r,s}:\mathcal{M}_{r,s} \to \mathcal{M}_{r,s}$ is then defined by
\begin{equation}\label{lap}
\Delta_{r,s} = -\overline{\partial}_{s-1}\partial_r + r(s-1) = - \partial_{r-1} \overline{\partial}_s + s(r-1).
\end{equation}
This induces a bigraded operator $\Delta$ of bidegree $(0,0)$ on $\mathcal{M}$. 

An example of a real analytic modular form (and a standard modular form) is the Eisenstein series. For an integer $r\geq 4$ and even, we define the Eisenstein series of weights $(r,0)$ by
\begin{equation}\label{Eisen}
    \mathbb{G}_{r}(z) = -\dfrac{B_r}{2r} + \sum^{\infty}_{n=1} \sigma_{r-1}(n)e^{2\pi i n z},
\end{equation}
where $B_r$ is the $r$-th Bernoulli number and $\sigma_{x}(n)$ is the divisor function, given by $\sum_{d\mid n} d^x$.

\subsection{Modular Iterated Integrals}\label{mii}

In this paper, rather than looking at the entirety of $\mathcal{M}$, we will focus on the subspace of modular iterated integrals, introduced by Brown in \cite{brown1}.
\begin{definition}
We let $\mathcal{MI}_{-1}=0$. For any $n \geq 0 \in \mathbb{Z}$ we define the space of modular iterated integrals of length $n$, $\mathcal{MI}_n$, to be the largest subspace of $\bigoplus_{r,s \geq 0} \mathcal{M}_{r,s}$ which satisfies 
\begin{align*}
\partial \mathcal{MI}_n &\subset \mathcal{MI}_n + M[\mathbb{L}] \times \mathcal{MI}_{n-1}, \\
\overline{\partial} \mathcal{MI}_n &\subset \mathcal{MI}_n + \overline{M}[\mathbb{L}] \times \mathcal{MI}_{n-1},
\end{align*}
where $M$ is the space of holomorphic modular forms. 
\end{definition}
When $n=0$ we have the following description of $\mathcal{MI}_0$.
\begin{lemma}(Lemma 3.10 of \cite{brown1})
\begin{equation*}
    \mathcal{MI}_0= \mathbb{C}[\mathbb{L}^{-1}].
\end{equation*}
\end{lemma}
We will also be able to provide a simple and concise description of $\mathcal{MI}_1$, but first, we will need to introduce the real analytic Eisenstein series.

\begin{definition}\label{ers}
Let $r,s\geq 0$ and $r+s=w \geq 1$, then the real analytic Eisenstein series of weights $(r,s)$ is given by
\begin{equation*}
    \mathcal{E}_{r,s}= \dfrac{w!}{2 \cdot (2\pi i)^{w+2}} \sum_{\substack{(m,n)\in \mathbb{Z}^2 \\ (m,n)\neq (0,0)}} \dfrac{\mathbb{L}}{(mz+n)^{r+1}(m\bar{z}+n)^{s+1}}.
\end{equation*}
\end{definition}
The following lemma provides us with some properties of the real analytic Eisenstein series.
\begin{lemma}(Proposition 4.1 of \cite{brown1})\label{de=lg}
Let $r,s$ and $w$ be as above, then $\mathcal{E}_{r,s}$ satisfies the following equations:
\begin{align*}
    \partial \mathcal{E}_{w,0} =\mathbb{L}\mathbb{G}_{w+2}&, \\
    \partial \mathcal{E}_{r,s} - (r+1)\mathcal{E}_{r+1,s-1}=0&,\ \hspace{10mm} \text{for } s\geq 1.
    \\[4mm]
    \overline{\partial} \mathcal{E}_{0,w} =\mathbb{L}\overline{\mathbb{G}}_{w+2}&, \\
    \overline{\partial} \mathcal{E}_{r,s} - (s+1)\mathcal{E}_{r-1,s+1}=0&,\ \hspace{10mm} \text{for } r\geq 1. 
    \\[4mm]
    \Delta \mathcal{E}_{r,s} = -(r+s)\mathcal{E}_{r,s}&.
\end{align*}
\end{lemma}

As mentioned above, the functions $\mathcal{E}_{r,s}$ can also be used to provide a description of $\mathcal{MI}_1$.

\begin{lemma}(Corollary 4.4 of \cite{brown1})
\begin{equation*}
    \mathcal{MI}_1 = \mathcal{MI}_0 \otimes_{\mathbb{C}} \bigoplus_{\substack{r,s \geq 0 \\ r+s \geq 2}} \  \mathbb{C} \,\mathcal{E}_{r,s}.
\end{equation*}

\end{lemma}

Therefore, we know of all the possible functions that can exist in $\mathcal{MI}_0$ and $\mathcal{MI}_1$. The same cannot be said, however, for higher lengths. Fortunately, the following lemma, which generalises Proposition 4.1 of \cite{brown1}, gives us one way of checking if a function is in $\mathcal{MI}_n$, for any $n$. This lemma will be useful when proving the main theorem of this paper. 

\begin{lemma}\label{thmFrs}
Let $r,s  \geq 0$ be integers and set $2w = r+s$. Suppose $F_{r,s} \in \mathcal{M}_{r,s}$ are a family of elements such that 
\begin{align}
	\partial F_{2w,0} &\in M[\mathbb{L}] \times \mathcal{MI}_{n-1}, \label{Frs1} \\
	\partial F_{r,s} - (r+1)F_{r+1,s-1} &\in M[\mathbb{L}] \times \mathcal{MI}_{n-1},\qquad \text{if } s\geq 1, \label{Frs2} 
	\end{align}
and
\begin{align}
	\overline\partial {F}_{0,2w} &\in \overline{M}[\mathbb{L}] \times \mathcal{MI}_{n-1}, \label{Frs3} \\
	\overline\partial {F}_{r,s} - (s+1)F_{r-1,s+1} &\in \overline{M}[\mathbb{L}] \times \mathcal{MI}_{n-1}, \qquad \text{if } r\geq 1 \label{Frs4}.
	\end{align}
Then the $F_{r,s} \in \mathcal{MI}_n$.
\end{lemma}
\begin{proof}
By the definition of $\mathcal{MI}_n$, any space $B \subset \bigoplus_{r,s \geq 0}  \mathcal{M}_{r,s}$  which satisfies 
\begin{align*}
	\partial B &\subseteq B + M[\mathbb{L}] \times \mathcal{MI}_{n-1},\\ 
	\overline{\partial} B &\subseteq B + \overline{M}[\mathbb{L}] \times \mathcal{MI}_{n-1} 
\end{align*}
must be contained within $\mathcal{MI}_{n} $. Let $B$ denote the space of the family of elements $F_{r,s}$ then, by equations (\ref{Frs1}) and (\ref{Frs2}),
\begin{align*}
	\partial F_{2w,0} \subseteq M[\mathbb{L}] \times \mathcal{MI}_{n-1}&,\\
	\partial F_{r,s} \subseteq B + M[\mathbb{L}] \times \mathcal{MI}_{n-1}&, \qquad \text{if } s\geq 1  ,
	\end{align*}
and hence $\partial B \subseteq B + M[\mathbb{L}] \times \mathcal{MI}_{n-1}$. Similarly, by equations (\ref{Frs3}) and (\ref{Frs4}), we have $\overline{\partial} B \subseteq B + \overline{M}[\mathbb{L}] \times \mathcal{MI}_{n-1}$. Therefore $B$ must be contained in $\mathcal{MI}_n$ due to the maximality of $\mathcal{MI}_n$ and we conclude that $F_{r,s} \in \mathcal{MI}_n$.
\end{proof}

\subsection{Representations of \texorpdfstring{$\Gamma_1$}{Gamma1}}

Let $n \geq 0$, then we define
\begin{equation*}
    V_{2n} = \bigoplus_{r+s=2n} X^rY^s\mathbb{Q},
\end{equation*}
which is equipped with the (single) slash operator defined by
\begin{equation*}
    (X,Y)|_{\gamma} = (aX+bY, cX+dY), \qquad \forall \ \gamma =\begin{psmallmatrix}
a &  b \\
c &  d 
\end{psmallmatrix} \in \Gamma_1.
\end{equation*}
For any $f \in \mathcal{M}_{r,0}$, we define the one form $\underline{f}$ by
\begin{equation*}
    \underline{f}(z) = 2 \pi i f(z) (X-zY)^{r-2} dz.
\end{equation*}
The following two lemmas will be extremely useful throughout this paper.
\begin{lemma}(Proposition 7.1 of \cite{brown1})\label{7.1}
Let $f: \mathcal{H} \to V_{2n} \otimes \mathbb{C}$ be real analytic, then it can be written in the form 
\begin{equation*}
    f = \sum_{r+s=2n} f_{r,s} (z) (X-zY)^r(X-\bar{z}Y)^s
\end{equation*}
such that $(z-\bar{z})^{2n}f_{r,s}:\mathcal{H}\to \mathbb{C}$ are real analytic. The function $f$ is modular equivariant:
\begin{equation}\label{modeq}
    f(\gamma(z))|_{\gamma} = f(z) , \qquad \forall \ \gamma \in \Gamma_1,
\end{equation}
if and only if every $f_{r,s}$ is modular of weights (r,s) (as defined by equation \eqref{modinv2}).
\end{lemma}
If any of the above $f_{r,s}$ are also of the form \eqref{f=aqq}, then they are real analytic modular forms of weights $(r,s)$. We now give a lemma involving the operators $\partial$ and $\bar{\partial}$.

\begin{lemma}(Proposition 7.2 of \cite{brown1})\label{dF=A}
Let $F,A,B:\mathcal{H} \to V_{2n} \otimes \mathbb{C}$ be real analytic, then the equation 
\begin{equation*}
    \dfrac{\partial F}{\partial z} = \pi i A(z)
\end{equation*}
is equivalent to the following system of equations (where $r,s  \geq0$ and $r+s=2n$):
\begin{align*}
    \partial F_{2n,0}&= \mathbb{L}A_{2n,0}, \\
    \partial F_{r,s}-(r+1)F_{r+1,s-1}&= \mathbb{L}A_{r,s}, \hspace{10mm} \text{for } s \geq 1.
\end{align*}
Similarly, the equation 
\begin{equation*}
    \dfrac{\partial F}{\partial \bar{z}} = -\pi i B(z)
\end{equation*}
is equivalent to the following system of equations (where $r,s  \geq0$ and $r+s=2n$):
\begin{align*}
    \overline{\partial} F_{0,2n}&= \mathbb{L}B_{0,2n}, \\
    \overline{\partial} F_{r,s}-(s+1)F_{r-1,s+1}&= \mathbb{L}B_{r,s}, \hspace{10mm} \text{for } r \geq 1.
\end{align*}
\end{lemma}

\section{Review of the Length Two Case}\label{trip}

The aim of this paper is to produce Laplace-eigenvalue equations involving ``new" functions in $\mathcal{MI}_3$; we will see how this is achieved in the next section. However, in order to best understand the underlying ideas of this theory, we first provide a review of the length two case given in Section 9 of \cite{brown1}. 

\subsection{Double Eisenstein Integrals}\label{doub}

Let $a> 0 \in \mathbb{Z}$, then we define 
\begin{align}
    \mathcal{E}_{2a} (z) &:= \sum_{r+s=2a} \mathcal{E}_{r,s}(X-z Y)^r(X-\bar{z}Y)^s, \label{e2a}
    \\
    E_{2a+2}(z) &:= 2 \pi i \, \mathbb{G}_{2a+2}(z)(X- zY)^{2a}, \label{E2a}
    \\[1mm]
     \underline{E}_{2a+2}(z) &:= E_{2a+2}(z) \, dz. \nonumber
\end{align}
The functions $\mathcal{E}_{2a}(z)$ and $\underline{E}_{2a+2}(z)$ are modular equivariant and closely linked to one another. By Lemmas \ref{de=lg} and \ref{dF=A}, we have
\begin{equation}\label{Elinks}
    \dfrac{\partial \mathcal{E}_{2a}(z)}{\partial z}  = \dfrac{1}{2} E_{2a+2}(z) \quad \text{ and } \quad  \dfrac{\partial \mathcal{E}_{2a}(z)}{\partial \bar{z}} = \dfrac{1}{2} \overline{E}_{2a+2} (z).
\end{equation}
For $a,b \geq 2$, we consider the one-forms
\begin{align*}
    D_{2a,2b}(z)&: \mathcal{H} \to \left(V_{2a-2} \otimes V_{2b-2} \right) \otimes \left(\mathbb{C} \, dz + \mathbb{C} \,d\bar z \right) \\
     D_{2a,2b}(z)&:= \underline{E}_{2a}(z) \otimes \mathcal{E}_{2b-2}(z)+ \mathcal{E}_{2a-2}(z) \otimes \overline{\underline{E}}_{2b}(z).
\end{align*}
They are modular equivariant:
\begin{equation*}
        D_{2a,2b}(\gamma z)|_{\gamma} = D_{2a,2b}(z), \qquad \forall \ \gamma\in \Gamma_1,
\end{equation*}
and, by Lemma 9.1 of \cite{brown1}, closed:
\begin{equation*}
    d D_{2a,2b} = 0.
\end{equation*}
We then consider the integral 
\begin{equation*}
    K_{2a,2b}(z) =  -\dfrac{1}{2} \int_{z}^{\reg} D_{2a,2b}(\tau).
\end{equation*}
This integral depends only on the homotopy class of the chosen path. The integral is regularised as in Section 8 of \cite{brown1} (see also, Section 4 of \cite{BrownMultiple}).
To see what this regularisation means we first write 
\begin{equation*}
    D_{2a,2b}(\tau) = D^{(0)}_{2a,2b}(\tau) + D^{(\infty)}_{2a,2b}(\tau),
\end{equation*}
such that $D^{(0)}_{2a,2b}(\tau)$ is the part of $D_{2a,2b}(\tau)$ that exponentially tends to zero as $\tau$ tends to $i\infty$.
Now the regularisation above can be viewed as
\begin{equation*}
    K_{2a,2b}(z) = -\dfrac{1}{2} \int^{\infty}_{z} D^{(0)}_{2a,2b}(\tau) + \dfrac{1}{2} \int^{z}_{0} D^{(\infty)}_{2a,2b}(\tau).
\end{equation*}
This allows us to see that $K_{2a,2b}(z)$ satisfies the following differential equations:
\begin{align}
    \dfrac{\partial}{\partial z} K_{2a,2b}(z) &= \pi i  \mathbb{G}_{2a}(z)(X-zY)^{2a-2}\otimes \mathcal{E}_{2b-2}(z)\label{dzK},\\[1mm]
     \dfrac{\partial}{\partial \bar{z}} K_{2a,2b}(z) &=  \mathcal{E}_{2a-2}(z)\otimes  \overline{\pi i\mathbb{G}_{2b}}(z)(X-\bar{z}Y)^{2b-2}. \label{dzbar}
\end{align}

For $k \geq 0$, there is a $\Gamma_1$ equivariant projector given by
\begin{align*}
    \delta^k &: V_{2a} \otimes V_{2b} \to V_{2a+2b-2k} \\
    \delta^k &:= m \circ \left( \dfrac{\partial}{\partial X} \otimes \dfrac{\partial}{\partial Y} - \dfrac{\partial}{\partial Y}\otimes \dfrac{\partial }{\partial X} \right)^k,
    \end{align*}
where $m:\mathbb{Q}[X,Y]\otimes \mathbb{Q}[X,Y] \to \mathbb{Q}[X,Y]$ is the multiplication map. We have the following lemma involving this projector.

\begin{lemma}\label{dk}(Lemma 7.3 of \cite{brown1}) Suppose that $D:\mathcal{H} \to V_{2n}$ and set
\begin{equation*}
    F = \dfrac{\delta^k}{(k!)^2} \Big( \mathbb{G}_{2m+2}(X-zY)^{2m} \otimes D   \Big).
\end{equation*}
Then $F: \mathcal{H} \to V_{2m+2n-2k}$ vanishes if either $k>2m$ or $k>2n$, and otherwise satisfies 
\begin{equation*}
    F_{r,s} = (z-\bar{z})^k \mathbb{G}_{2m+2}\binom{2m}{k}\binom{s+k}{k}D_{r-2m+k,s+k},
\end{equation*}
where we set $D_{r,s}=0$ if either $r$ or $s<0$.
\end{lemma}
Now, for $a,b \geq 1$ and $0 \leq k \leq \text{min}\{2a,2b\}$, we define 
\begin{equation*}
    K^{(k)}_{2a+2,2b+2}(z) = (\pi i)^k \frac{\delta^k}{(k!)^2} K_{2a+2,2b+2}(z)
\end{equation*}
and observe that
\begin{align*}
    \dfrac{\partial}{\partial z}K^{(k)}_{2a+2,2b+2}(z) &= \dfrac{\partial}{\partial z} \left( (\pi i)^k \frac{\delta^k}{(k!)^2} K_{2a+2,2b+2}(z)   \right) \\
    &= (\pi i)^k \frac{\delta^k}{(k!)^2}  \left(\dfrac{\partial}{\partial z}  K_{2a+2,2b+2}(z)   \right).
\end{align*}
Then, using equation \eqref{dzK} and Lemma \ref{dk}, this becomes
\begin{equation}\label{dK2}
    \dfrac{\partial}{\partial z}K^{(k)}_{2a+2,2b+2}(z) \,= \, \pi i J^{(k)}(z)
\end{equation}
where
\begin{equation}\label{dK3}
   J^{(k)}_{r,s}(z) =  \mathbb{L}^k \mathbb{G}_{2a+2}(z) \binom{2a}{k}\binom{k+s}{k}  \mathcal{E}_{r-2a+k,s+k}(z).
\end{equation}

\subsection{Equivariant Versions of Double Eisenstein Integrals}\label{EVTEI}
The function $K^{(k)}_{2a+2,2b+2}$ is not quite modular equivariant. (We want an equivariant function in order to make use of Lemma \ref{7.1} during the proof of Theorem \ref{MI3THM}. This is covered in more detail in Section \ref{length3}, specifically the proof of Theorem \ref{3MI3THM}). Another way to view the equation of modular equivariance \eqref{modeq} is
\begin{equation}
    f(\gamma(z))|_\gamma -f(z) = 0, \qquad \forall \ \gamma \in \Gamma_1.
\end{equation}
We can also deduce from equations \eqref{dzK} and \eqref{dzbar} that $dK_{2a,2b}$ is modular equivariant. It follows that, $\forall \ \gamma  \in \Gamma_1$, 
\begin{align*}
    K_{2a,2b}(\gamma z)|_{\gamma}-K_{2a,2b}(z) =c_{\gamma},
\end{align*}
for some ``constant'' $c_{\gamma}$ (does not depend on $z$) which defines a coboundary and hence a cocycle. Combining this with the fact that $\delta^k$ is an equivariant projector, we conclude that
\begin{align*}
K^{(k)}_{2a,2b}(\gamma z)|_{\gamma}-K^{(k)}_{2a,2b}(z)=C_{\gamma},
\end{align*}
for a ``constant $C_{\gamma}$'' (does not depend on $z$) which again defines a cocycle:
\begin{equation*}
    \gamma \mapsto C_{\gamma} \in Z^1(\Gamma_1,V_{2a+2b-4-2k} \otimes \mathbb{C}).
\end{equation*}

By the Eichler-Shimura theorem, such a cocycle can be written as a linear combination of cocycles of cusp forms, the complex conjugate of said forms, Eisenstein series and a coboundary $c|_{\gamma - \text{id}}$ for a $c \in V_{2a+2b-4-2k}\otimes \mathbb{C}$. We then define a function $M^{(k)}_{2a,2b}$ such that we have modular equivariance and 
\begin{equation}\label{MM}
    M^{(k)}_{2a,2b} = K^{(k)}_{2a,2b} - c - \dfrac{1}{2} \int^{\reg}_{z} (\underline{f} + \overline{\underline{g}}),
\end{equation}
where $f$ (resp. $g$) is a standard modular form (resp. modular cusp form) of weight $2a+2b-2-2k$. The equation above uniquely determines $M^{(k)}_{2a,2b}, c, f$ and $g$ except for when $2a+2b-4-2k=0$, as we could add an arbitrary constant $d \in \mathbb{C}$.

\begin{theorem}\label{MI3THM}(Theorem 9.3 of \cite{brown1})
Let $a,b \geq 1$, $0\leq k \leq \text{min}\{2a,2b  \}$ and set $w=a+b-k$. There exists a family of elements $(F^{(k)}_{2a+2,2b+2})_{r,s} \in \mathcal{MI}_2 \cap \mathcal{M}_{r,s} $ of total modular weight $2w = r+s$, with $r,s \geq 0$, which satisfy the following:
\begin{align}
     a)
     \ \,  & \partial (F^{(k)}_{2a+2,2b+2})_{2w,0} = \binom{2a}{k} \mathbb{L}^{k+1}\mathbb{G}_{2a+2}\mathcal{E}_{2b-k,k} + \mathbb{L}f, \label{MI3THM1} \\[1mm]
& \partial (F^{(k)}_{2a+2,2b+2})_{r,s} - (r+1)(F^{(k)}_{2a+2,2b+2})_{r+1,s-1} \label{MI3THM2} \\ & \hspace{34mm}= \binom{2a}{k}\binom{k+s}{k} \mathbb{L}^{k+1}\mathbb{G}_{2a+2}\mathcal{E}_{2b-k-s,k+s}   , \tag*{\quad if $s \geq 1$,} \nonumber
\end{align}
where $f$ is a unique cusp form of weight $2w+2$.
\begin{align}
      b) \ \, & \overline{\partial} (F^{(k)}_{2a+2,2b+2})_{0,2w} = \binom{2b}{k} \mathbb{L}^{k+1}\overline{\mathbb{G}}_{2b+2}\mathcal{E}_{k,2a-k}
+ \mathbb{L}\overline{g}, \label{thmk11}  \\[1mm]
   & \overline{\partial} (F^{(k)}_{2a+2,2b+2})_{r,s} - (s+1)(F^{(k)}_{2a+2,2b+2})_{r-1,s+1}  \label{thmk12} \\ & \hspace{34mm}= \binom{2b}{k}\binom{k+r}{k} \mathbb{L}^{k+1}\overline{\mathbb{G}}_{2b+2}\mathcal{E}_{k+r,2a-k-r}, \nonumber\tag*{\quad if  $r \geq 1$,}
\end{align}
where $g$ is a unique modular form of weight $2w+2$.
\end{theorem}
We note that, in the above theorem, $f$ is a cusp form and $g$ is a modular form but in equation \eqref{MM} it is the opposite. In this theorem, we actually have the freedom to chose either $f$ or $g$ to be the guaranteed cusp form (but not both) and it is the author's personal preference to choose $f$. The reason we can make such a choice is given in the proof of Theorem 9.3 of \cite{brown1}.

The function 
\begin{equation}\label{Fkab}
    F^{(k)}_{2a+2,2b+2}(z) := \sum_{\mathclap{\substack{r+s= \\
    2a+2b-2k}}} \, \big(F^{(k)}_{2a+2,2b+2}(z)\big)_{r,s}(X-zY)^r(X-\bar{z}Y)^s
\end{equation}
is modular equivariant and satisfies the following differential equations:
\begin{align}
   & \dfrac{\partial }{\partial z} F^{(k)}_{2a+2,2b+2}(z)= (i \pi)^{k}\dfrac{\delta^{(k)}}{(k!)^2}\Big(\dfrac{1}{2}E_{2a+2}(z) \otimes \mathcal{E}_{2b}(z)\Big) \label{dFdz} \\
   & \hspace{56mm} + \pi i f^{(k)}_{2a+2,2b+2}(z)(X-zY)^{2a+2b-2k}, \nonumber
\\[2mm]
   &  \dfrac{\partial }{\partial \bar{z}} F^{(k)}_{2a+2,2b+2}(z) = (i \pi)^{k}\dfrac{\delta^{(k)}}{(k!)^2}\Big(\mathcal{E}_{2a}(z) \otimes \dfrac{1}{2}\overline{E}_{2b+2}(z) \Big)  \label{dFdz2} \\
   & \hspace{56mm} - \pi i \bar{g}^{(k)}_{2a+2,2b+2}(z)(X-\bar{z}Y)^{2a+2b-2k}, \nonumber
\end{align}
where $f^{(k)}_{2a+2,2b+2}$ and $\bar{g}^{(k)}_{2a+2,2b+2}$ are the  modular forms determined by Theorem \ref{MI3THM} and are of weight $2a+2b-2k+2$. It is well known that there exist no cusp forms of weight $\leq 10$ or weight $14$.
Therefore, equation \eqref{dFdz} does not include the summand containing $f$ when $2a+2b-2k \leq 8$ or $2a+2b-2k = 12$. Furthermore, under these conditions, we can write $\overline{g}^{(k)}_{2a+2,2b+2}$ as $C^{(k)}_{2a+2,2b+2}\overline{\mathbb{G}}_{2w+2}$, for some constant $C^{(k)}_{2a+2,2b+2}$.


%

These differential equations will be crucial in Section \ref{length3} when we construct an analogous theorem to Theorem \ref{MI3THM}, but for functions in $\mathcal{MI}_3$ rather than $\mathcal{MI}_2$.

\subsection{Length Two Examples}\label{examples}
\subsubsection{Example 1}\label{ex1}

The simplest example is when $2a=2b=2$, first given by Brown in \cite{brown1}. Since $2a+2b-2k \leq 4$, for $k \in \{0,1,2\}$, the cusp form given in equation \eqref{MI3THM1} vanishes. We are also given, in \cite{brown1}, that each $\overline{g}^{(k)}_{4,4}$ vanishes.

When $k=0$, we have the following set of equations from Theorem \ref{MI3THM}:
\begin{align*}
&\partial (F^{(0)}_{4,4})_{4,0}=\mathbb{L}\mathbb{G}_4 \mathcal{E}_{2,0}
\hspace{10mm} 
&&\overline{\partial} (F^{(0)}_{4,4})_{0,4} =  \mathbb{L}\overline{\mathbb{G}}_4 \mathcal{E}_{0,2}
\\[1mm]
&\partial (F^{(0)}_{4,4})_{3,1} - 4(F^{(0)}_{4,4})_{4,0}=\mathbb{L}\mathbb{G}_4 \mathcal{E}_{1,1}
\hspace{10mm}  
&&\overline{\partial}(F^{(0)}_{4,4})_{1,3} - 4(F^{(0)}_{4,4})_{0,4} = \mathbb{L}\overline{\mathbb{G}}_4 \mathcal{E}_{1,1}
\\[1mm]
&\partial (F^{(0)}_{4,4})_{2,2} - 3(F^{(0)}_{4,4})_{3,1}=\mathbb{L}\mathbb{G}_4 \mathcal{E}_{0,2}
\hspace{10mm}  
&&\overline{\partial} (F^{(0)}_{4,4})_{2,2} - 3(F^{(0)}_{4,4})_{1,3}= \mathbb{L}\overline{\mathbb{G}}_4 \mathcal{E}_{2,0}
\\[1mm]
& \partial (F^{(0)}_{4,4})_{1,3} - 2(F^{(0)}_{4,4})_{2,2}=0 
\hspace{10mm}  
&&\overline{\partial} (F^{(0)}_{4,4})_{3,1} - 2(F^{(0)}_{4,4})_{2,2}=0
\\[1mm]
& \partial (F^{(0)}_{4,4})_{0,4} - (F^{(0)}_{4,4})_{1,3}=0   
\hspace{10mm}  
&&\overline{\partial} (F^{(0)}_{4,4})_{4,0}-(F^{(0)}_{4,4})_{3,1}=0.
\end{align*} 

When k=1:
\begin{align*}
&\partial (F^{(1)}_{4,4})_{2,0}=2\mathbb{L}^2\mathbb{G}_4 \mathcal{E}_{1,1} 
\hspace{10mm} 
&& \overline{\partial} (F^{(1)}_{4,4})_{0,2} = 2\mathbb{L}^2\overline{\mathbb{G}}_4\mathcal{E}_{1,1}
\\[1mm]
&\partial (F^{(1)}_{4,4})_{1,1} - 2(F^{(1)}_{4,4})_{2,0}=4\mathbb{L}^2\mathbb{G}_4 \mathcal{E}_{0,2}  
\hspace{10mm} 
&& \overline{\partial} (F^{(1)}_{4,4})_{1,1} - 2(F^{(1)}_{4,4})_{0,2} = 4\mathbb{L}^2\overline{\mathbb{G}}_4\mathcal{E}_{2,0}
\\[1mm]
&\partial (F^{(1)}_{4,4})_{0,2} - (F^{(1)}_{4,4})_{1,1}= 0 
\hspace{10mm} 
&& \overline{\partial} (F^{(1)}_{4,4})_{2,0} - (F^{(1)}_{4,4})_{1,1} = 0.
\end{align*} 

Finally, when $k=2$ we have two equations:
\begin{align*}
    &\partial (F^{(2)}_{4,4})_{0,0} = \mathbb{L}^3 \mathbb{G}_4 \mathcal{E}_{0,2} \hspace{20mm}
    & \overline{\partial} (F^{(2)}_{4,4})_{0,0} = \mathbb{L}^3 \overline{\mathbb{G}}_4 \mathcal{E}_{2,0}.
    \end{align*}
For $k=0$, the equations above can be solved in terms of real analytic Eisenstein series:
\begin{align*}
&(F^{(0)}_{4,4})_{4,0} =\dfrac{1}{2}\mathcal{E}_{2,0}\mathcal{E}_{2,0}
\\
&(F^{(0)}_{4,4})_{3,1} =\mathcal{E}_{2,0}\mathcal{E}_{1,1}
&&(F^{(0)}_{4,4})_{1,3} =\mathcal{E}_{1,1}\mathcal{E}_{0,2}
\\
&(F^{(0)}_{4,4})_{2,2} =\mathcal{E}_{2,0}\mathcal{E}_{0,2}+ \dfrac{1}{2}\mathcal{E}_{1,1}\mathcal{E}_{1,1}
&& (F^{(0)}_{4,4})_{0,4} = \dfrac{1}{2}\mathcal{E}_{0,2}\mathcal{E}_{0,2}.
\end{align*} 
This is also true for $k=2$:
\begin{align*}
    (F^{(2)}_{4,4})_{0,0} = \mathbb{L}^2\mathcal{E}_{2,0}\mathcal{E}_{0,2} - \dfrac{1}{4}\mathbb{L}^2\mathcal{E}_{1,1}\mathcal{E}_{1,1}.
\end{align*}
However, we cannot do the same for the functions $(F^{(1)}_{4,4})_{r,s}$, we consider these functions to be ``new''.
For each of the functions $(F^{(k)}_{4,4})_{r,s}$ above, we have an associated Laplace-eigenvalue equation:
\begin{align*}
& (\Delta +4)(F^{(0)}_{4,4})_{4,0}= - \mathbb{L}\mathbb{G}_4 \mathcal{E}_{1,1} \hspace{12mm}
&&  (\Delta+2)(F^{(1)}_{4,4})_{2,0} = -4\mathbb{L}^2\mathbb{G}_4\mathcal{E}_{0,2}
\\
& (\Delta +4)(F^{(0)}_{4,4})_{3,1}= -2 \mathbb{L}\mathbb{G}_4 \mathcal{E}_{0,2}
&& (\Delta+2)(F^{(1)}_{4,4})_{1,1} = -4\mathbb{L}^3\mathbb{G}_4\overline{\mathbb{G}}_4 
\\
& (\Delta +4)(F^{(0)}_{4,4})_{2,2}= - \mathbb{L}^2\mathbb{G}_4 \overline{\mathbb{G}}_4
&& (\Delta+2)(F^{(1)}_{4,4})_{0,2} = -4\mathbb{L}^2\overline{\mathbb{G}}_4\mathcal{E}_{2,0}
\\
& (\Delta +4)(F^{(0)}_{4,4})_{1,3}= -2 \mathbb{L}\overline{\mathbb{G}}_4 \mathcal{E}_{2,0}
\\
& (\Delta +4)(F^{(0)}_{4,4})_{0,4}= - \mathbb{L}\overline{\mathbb{G}}_4 \mathcal{E}_{1,1}
&& \Delta (F^{(2)}_{4,4})_{0,0}= - \mathbb{L}^4\mathbb{G}_4\overline{\mathbb{G}}_4.
\end{align*} 

One use of these ``new'' functions, given in \cite{brown1}, is that they should allow us to give new representations of modular graph functions. To see this, we first observe the following:
\begin{align*}
    &(\Delta + 2)(\mathbb{L}^2\mathcal{E}_{2,0}\mathcal{E}_{0,2}) = - \mathbb{L}^4\mathbb{G}_4\overline{\mathbb{G}}_4- \mathbb{L}^2\mathcal{E}_{1,1}\mathcal{E}_{1,1}, \\[1mm]
    &(\Delta+2)\mathbb{L}^3\mathcal{E}_{3,3}=-10\mathbb{L}^3\mathcal{E}_{3,3}.
\end{align*}
Therefore, we have
\begin{equation*}
    \big(\Delta + 2\big)\Big(4\mathbb{L}(F^{(1)}_{4,4})_{1,1} -16\mathbb{L}^2\mathcal{E}_{2,0}\mathcal{E}_{0,2} + \dfrac{1}{25}\mathbb{L}^3\mathcal{E}_{3,3} \Big) = 16\mathbb{L}^2\mathcal{E}_{1,1}\mathcal{E}_{1,1}- \dfrac{2}{5}\mathbb{L}^3\mathcal{E}_{3,3}
\end{equation*}
and we conclude that the modular graph function $C_{2,1,1}$, which satisfies 
\begin{equation}\label{c211}
    (\Delta + 2)C_{2,1,1} = 16\mathbb{L}^2\mathcal{E}_{1,1}\mathcal{E}_{1,1}- \dfrac{2}{5}\mathbb{L}^3\mathcal{E}_{3,3},
\end{equation}
can be expressed in terms of $\mathbb{L}(F^{(1)}_{4,4})_{1,1}, \mathbb{L}^2\mathcal{E}_{2,0}\mathcal{E}_{0,2}$, $\mathbb{L}^3\mathcal{E}_{3,3}$ and a constant. The coefficients in the expansion \eqref{f=aqq} of $(F^{(1)}_{4,4})_{1,1}$ can be determined using properties of $\partial$, $\overline{\partial}$ and the above differential equations (see \cite{brown1} for more details).

We can use Theorem \ref{MI3THM} to produce more ``new'' functions and, in turn, to produce more Laplace-eigenvalue equations. By comparing these equations with data on modular graph functions from literature in physics, we should be able to repeat the above result for other modular graph functions. 

Of course, it will be useful to have as many Laplace-eigenvalue equations available to use as possible. Indeed, it may be useful, then, to have Laplace-eigenvalue equations for ``new'' functions that are of a higher length than two. We will deal with the length three case in Section \ref{length3}.

\subsubsection{A new description of ``new''}

In this paper we focus on when $2a+2b-2k \leq 8$ or $=12$, so that the cusp form $f^{(k)}_{2a+2,2b+2}$ vanishes and  we can write $\bar{g}^{(k)}_{2a+2,2b+2}$ as $C^{(k)}_{2a+2,2b+2}\overline{\mathbb{G}}_{2w+2}$. It would be useful to be able to determine the value of $C^{(k)}_{2a+2,2b+2}$ for specific $a$, $b$ and $k$, but the author is not currently aware of any examples of this being achieved other than the case $2a=2b=2$. 

Until we know more about this, we must be aware of the effect that $C^{(k)}_{2a+2,2b+2}  \overline{\mathbb{G}}_{2w+2}$ may have on our definition of a ``new'' function. Suppose that we did not know that $\bar{g}^{(k)}_{4,4}$ vanished in the previous example, then the $\partial, \bar{\partial}$ equations for $k=0$ would be
\begin{align*}
&\partial (F^{(0)}_{4,4})_{4,0}=\mathbb{L}\mathbb{G}_4 \mathcal{E}_{2,0}
\hspace{10mm} 
&&\overline{\partial} (F^{(0)}_{4,4})_{0,4} =  \mathbb{L}\overline{\mathbb{G}}_4 \mathcal{E}_{0,2} + C^{(0)}_{4,4}\LL\overline{\mathbb{G}}_{6}
\\
&\partial (F^{(0)}_{4,4})_{3,1} - 4(F^{(0)}_{4,4})_{4,0}=\mathbb{L}\mathbb{G}_4 \mathcal{E}_{1,1}
\hspace{10mm}  
&&\overline{\partial}(F^{(0)}_{4,4})_{1,3} - 4(F^{(0)}_{4,4})_{0,4} = \mathbb{L}\overline{\mathbb{G}}_4 \mathcal{E}_{1,1}
\\
&\partial (F^{(0)}_{4,4})_{2,2} - 3(F^{(0)}_{4,4})_{3,1}=\mathbb{L}\mathbb{G}_4 \mathcal{E}_{0,2}
\hspace{10mm}  
&&\overline{\partial} (F^{(0)}_{4,4})_{2,2} - 3(F^{(0)}_{4,4})_{1,3}= \mathbb{L}\overline{\mathbb{G}}_4 \mathcal{E}_{2,0}
\\
& \partial (F^{(0)}_{4,4})_{1,3} - 2(F^{(0)}_{4,4})_{2,2}=0 
\hspace{10mm}  
&&\overline{\partial} (F^{(0)}_{4,4})_{3,1} - 2(F^{(0)}_{4,4})_{2,2}=0
\\
& \partial (F^{(0)}_{4,4})_{0,4} - (F^{(0)}_{4,4})_{1,3}=0   
\hspace{10mm}  
&&\overline{\partial} (F^{(0)}_{4,4})_{4,0}-(F^{(0)}_{4,4})_{3,1}=0,
\end{align*} 
where $C^{(0)}_{4,4}$ is an unknown, potentially non-zero, constant. These equations, in general, would then no longer be solvable in terms of real analytic Eisenstein series, and the $(F^{(0)}_{4,4})_{r,s}$ could be considered ``new'' functions. The ability to solve these equations relied on knowing $C^{(0)}_{4,4}$ vanished. However, we note that the associated Laplace equations remain unchanged, no matter the value of $C^{(0)}_{4,4}$. Therefore, the modular form $\bar{g}^{(k)}_{2a+2,2b+2}$ can affect whether a function is considered ``new'' without affecting the Laplace equations.

In this paper, we are more interested in the Laplace equations than the $\partial, \bar{\partial}$ equations. With this in mind, we slightly redefine our definition of a ``new'' function to deal with the potential appearance of $C^{(k)}_{2a+2,2b+2} \LL \overline{\mathbb{G}}_{2w+2}$. 

Specifically, for $2a+2b+2k \leq 8$ or $2a+2b+2k=12$, we define the function $(F^{(k)}_{2a+2,2b+2})_{r,s}$ to be ``new'' if its associated $\partial, \bar{\partial}$ equations cannot be solved in terms of real analytic Eisenstein series when $C^{(k)}_{2a+2,2b+2}$ is set to zero. For example, even if we did not know $\bar{g}^{(k)}_{4,4}$ vanished in the previous example, we would still not consider the $(F^{(0)}_{4,4})_{r,s}$ functions to be ``new''. We will look at this definition in the context of another example.

\subsubsection{Example 2}

The next example we will look at is when $2a=2$ and $2b=4$. Again, since $2a+2b-2k \leq 6$, for $k \in \{0,1,2\}$, the cusp form given in equation \eqref{MI3THM1} vanishes. 

When $k=0$, we have the following set of equations:
\begin{align*}
&\partial (F^{(0)}_{4,6})_{6,0}=\mathbb{L}\mathbb{G}_4 \mathcal{E}_{4,0}
&&\overline{\partial}(F^{(0)}_{4,6})_{0,6}= \mathbb{L}\overline{\mathbb{G}}_6\mathcal{E}_{0,2}
+ C^{(0)}_{4,6}\LL\overline{\mathbb{G}}_{8}
\\
&\partial (F^{(0)}_{4,6})_{5,1} - 6(F^{(0)}_{4,6})_{6,0}=\mathbb{L}\mathbb{G}_4 \mathcal{E}_{3,1} \qquad
&& \overline{\partial}(F^{(0)}_{4,6})_{1,5} - 6(F^{(0)}_{4,6})_{0,6} = \mathbb{L}\overline{\mathbb{G}}_6\mathcal{E}_{1,1}
\\
&\partial (F^{(0)}_{4,6})_{4,2} - 5(F^{(0)}_{4,6})_{5,1}=\mathbb{L}\mathbb{G}_4 \mathcal{E}_{2,2}
&& \overline{\partial}(F^{(0)}_{4,6})_{2,4} - 5(F^{(0)}_{4,6})_{1,5} = \mathbb{L}\overline{\mathbb{G}}_6\mathcal{E}_{2,0}
\\
&\partial (F^{(0)}_{4,6})_{3,3} - 4(F^{(0)}_{4,6})_{4,2}= \mathbb{L}\mathbb{G}_4 \mathcal{E}_{1,3}
&& \overline{\partial}(F^{(0)}_{4,6})_{3,3} - 4(F^{(0)}_{4,6})_{2,4} = 0
\\
&\partial (F^{(0)}_{4,6})_{2,4} - 3(F^{(0)}_{4,6})_{3,3}= \mathbb{L}\mathbb{G}_4 \mathcal{E}_{0,4}
&& \overline{\partial}(F^{(0)}_{4,6})_{4,2} - 3(F^{(0)}_{4,6})_{3,3} = 0
\\
& \partial (F^{(0)}_{4,6})_{1,5} - 2(F^{(0)}_{4,6})_{2,4}=0 
&& \overline{\partial}(F^{(0)}_{4,6})_{5,1} - 2(F^{(0)}_{4,6})_{4,2} = 0 
\\
& \partial (F^{(0)}_{4,6})_{0,6} - (F^{(0)}_{4,6})_{1,5}=0
&& \overline{\partial}(F^{(0)}_{4,6})_{6,0} - (F^{(0)}_{4,6})_{5,1} = 0,
\end{align*} 
for some constant $C^{(0)}_{4,6}$.
When $k=1$:
\begin{align*}
&\partial (F^{(1)}_{4,6})_{4,0}=2\mathbb{L}^2\mathbb{G}_4 \mathcal{E}_{3,1} 
&& \overline{\partial}(F^{(1)}_{4,6})_{0,4}= 4\mathbb{L}^2\overline{\mathbb{G}}_6\mathcal{E}_{1,1} + C^{(1)}_{4,6}\LL\overline{\mathbb{G}}_{6}
\\
&\partial (F^{(1)}_{4,6})_{3,1} - 4(F^{(1)}_{4,6})_{4,0}=4\mathbb{L}^2\mathbb{G}_4 \mathcal{E}_{2,2} \quad
&& \overline{\partial}(F^{(1)}_{4,6})_{1,3} - 4(F^{(1)}_{4,6})_{0,4}= 8\mathbb{L}^2\overline{\mathbb{G}}_6\mathcal{E}_{2,0}
\\
&\partial (F^{(1)}_{4,6})_{2,2} - 3(F^{(1)}_{4,6})_{3,1}= 6\mathbb{L}^2\mathbb{G}_4 \mathcal{E}_{1,3} 
&& \overline{\partial}(F^{(1)}_{4,6})_{2,2} - 3(F^{(1)}_{4,6})_{1,3}=0
 \\
 & \partial (F^{(1)}_{4,6})_{1,3} - 2(F^{(1)}_{4,6})_{2,2}= 8\mathbb{L}^2\mathbb{G}_4 \mathcal{E}_{0,4}  
 && \overline{\partial}(F^{(1)}_{4,6})_{3,1} - 2(F^{(1)}_{4,6})_{2,2}=0
 \\
 &  \partial (F^{(1)}_{4,6})_{0,4} - (F^{(1)}_{4,6})_{1,3}=0 
 && \overline{\partial}(F^{(1)}_{4,6})_{4,0} - (F^{(1)}_{4,6})_{3,1}=0,
\end{align*} 

for some constant $C^{(1)}_{4,6}$. When $k=2$:
\begin{align*}
&\partial (F^{(2)}_{4,6})_{2,0}=\mathbb{L}^3\mathbb{G}_4 \mathcal{E}_{2,2} 
&&\overline{\partial}(F^{(2)}_{4,6})_{0,2}= 6\mathbb{L}^3\overline{\mathbb{G}}_6\mathcal{E}_{2,0} + C^{(2)}_{4,6}\LL\overline{\mathbb{G}}_{4}
\\
&\partial (F^{(2)}_{4,6})_{1,1} - 2(F^{(2)}_{4,6})_{2,0}=3\mathbb{L}^3\mathbb{G}_4 \mathcal{E}_{1,3}
&&\overline{\partial}(F^{(2)}_{4,6})_{1,1} - 2(F^{(2)}_{4,6})_{0,2}= 0
\\
&\partial (F^{(2)}_{4,6})_{0,2} - (F^{(2)}_{4,6})_{1,1}=6\mathbb{L}^3\mathbb{G}_4 \mathcal{E}_{0,4}
&&\overline{\partial}(F^{(2)}_{4,6})_{2,0} - (F^{(2)}_{4,6})_{1,1}= 0,
\end{align*} 
for some constant $C^{(2)}_{4,6}$.
In this example, even when $C^{(0)}_{4,6}$,  $C^{(1)}_{4,6}$ and $C^{(2)}_{4,6}$ are set to zero, none of the above equations can be solved in terms of real analytic Eisenstein series.  These functions are all ``new''. The associated Laplace-eigenvalue equations for these functions are:
\begin{align*}
    &(\Delta+6)(F^{(0)}_{4,6})_{6,0} =  -\mathbb{L}\mathbb{G}_4\mathcal{E}_{3,1}
    && (\Delta+4)(F^{(1)}_{4,6})_{4,0} =  -4\mathbb{L}^2\mathbb{G}_4\mathcal{E}_{2,2}
    \\
    &(\Delta+6)(F^{(0)}_{4,6})_{5,1}= -2\mathbb{L}\mathbb{G}_4\mathcal{E}_{2,2}\hspace{8mm}   \hspace{10mm}
    \
    && (\Delta+4)(F^{(1)}_{4,6})_{3,1}=  -12\mathbb{L}^2\mathbb{G}_4\mathcal{E}_{1,3}
    \\
    &(\Delta+6)(F^{(0)}_{4,6})_{4,2}= -3\mathbb{L}\mathbb{G}_4\mathcal{E}_{1,3}   \hspace{10mm}
    && (\Delta+4)(F^{(1)}_{4,6})_{2,2}= -24\mathbb{L}^2\mathbb{G}_4\mathcal{E}_{0,4}
    \\
    &(\Delta+6)(F^{(0)}_{4,6})_{3,3}= -4\mathbb{L}\mathbb{G}_4\mathcal{E}_{0,4}   \hspace{10mm}
    && (\Delta+4)(F^{(1)}_{4,6})_{1,3}= -8\mathbb{L}^3\mathbb{G}_4\overline{\mathbb{G}}_6
    \\
    & (\Delta+6)(F^{(0)}_{4,6})_{2,4}= - \mathbb{L}^2\mathbb{G}_4\overline{\mathbb{G}}_6
    && (\Delta+4)(F^{(1)}_{4,6})_{0,4}= -8\mathbb{L}^2\overline{\mathbb{G}}_6\mathcal{E}_{2,0}
    \\
    & (\Delta+6)(F^{(0)}_{4,6})_{1,5}= -2\mathbb{L}\overline{\mathbb{G}}_6\mathcal{E}_{2,0}
    &&   (\Delta+2)(F^{(2)}_{4,6})_{2,0} = -3\mathbb{L}^3\mathbb{G}_4\mathcal{E}_{1,3}
    \\
    & (\Delta+6)(F^{(0)}_{4,6})_{0,6}= -\mathbb{L}\overline{\mathbb{G}}_6\mathcal{E}_{1,1}
    && (\Delta+2)(F^{(2)}_{4,6})_{1,1}= -12\mathbb{L}^3\mathbb{G}_4\mathcal{E}_{0,4}
    \\
    &\hspace{10mm} && 
    (\Delta+2)(F^{(2)}_{4,6})_{0,2}= -6\mathbb{L}^4\mathbb{G}_4\overline{\mathbb{G}}_6.
\end{align*}

From Section \ref{sectmgf}, we have the equation
\begin{equation}\label{c311}
    (\Delta + 6)C_{3,1,1} + 3C_{2,2,1} = \dfrac{32}{3}\mathbb{L}^3\mathcal{E}_{1,1}\mathcal{E}_{2,2} - \dfrac{8}{315} \mathbb{L}^4\mathcal{E}_{4,4}.
\end{equation}
We will now show that we can express the RHS of this equation in terms of a Laplace equation for length two modular iterated integrals. (By equation \eqref{c221}, we know that $C_{2,2,1}$ can be expressed in terms of a real analytic Eisenstein series and a constant.) We first observe the following two equations:
\begin{align*}
    & (\Delta+6)(\mathbb{L}^3\mathcal{E}_{2,0}\mathcal{E}_{1,3}) = - 4 \mathbb{L}^4\mathbb{G}_4\mathcal{E}_{0,4} -  2\LL^3\mathcal{E}_{1,1}\mathcal{E}_{2,2}, \\[1mm]
    & (\Delta+6) \mathbb{L}^4\mathcal{E}_{4,4} = -14\mathbb{L}^4\mathcal{E}_{4,4}.
\end{align*}
Next, in Lemma 2.12 of \cite{brown1} we are given the following equation:
\begin{equation*}
    \Delta \mathbb{L} f = \mathbb{L} \Delta f - (p + q)\mathbb{L} f,
\end{equation*}
where $f$ is a real analytic modular form of weights $(p+1, q+1)$. We can combine this with the above Laplace equation $(\Delta+4)(F^{(1)}_{4,6})_{2,2} = -24\mathbb{L}^2\mathbb{G}_4\mathcal{E}_{0,4}$ to obtain
\begin{equation*}
    (\Delta+6)\mathbb{L}^2(F^{(1)}_{4,6})_{2,2} = -24\mathbb{L}^4\mathbb{G}_4\mathcal{E}_{0,4}.
\end{equation*}
Therefore, we have
\begin{align*}
    (\Delta+6)\left( \dfrac{4}{2205}\mathbb{L}^4\mathcal{E}_{4,4} - \dfrac{16}{3}\mathbb{L}^3\mathcal{E}_{2,0}\mathcal{E}_{1,3} + \dfrac{8}{9}\mathbb{L}^2(F^{(1)}_{4,6})_{2,2}    \right)& \\ = \dfrac{32}{3}\mathbb{L}^3\mathcal{E}_{1,1}\mathcal{E}_{2,2}& - \dfrac{8}{315} \mathbb{L}^4\mathcal{E}_{4,4}
\end{align*}
and we conclude that
\begin{align*}
    (\Delta\!+\!6)\!\left( \dfrac{4}{2205}\mathbb{L}^4\mathcal{E}_{4,4} - \dfrac{16}{3}\mathbb{L}^3\mathcal{E}_{2,0}\mathcal{E}_{1,3} + \dfrac{8}{9}\mathbb{L}^2(F^{(1)}_{4,6})_{2,2}   \! \right) \! =  (\Delta\! + \!6)C_{3,1,1} + 3C_{2,2,1}.
\end{align*}

In \cite{OLIAXEL}, the authors define modular invariant functions $F^{+(s)}_{m,n}$ and $F^{-(s)}_{m,n}$ for certain positive integers $m$, $n$ and $s$ (described below). These functions are extremely interesting for a variety of reasons. For example, not only do they determine all two-loop modular graph functions $C_{a,b,c}$ (and more), they actually transcend the space of modular graph functions itself. They are also expected to be very closely related to the functions $(F^{(k)}_{2a+2,2b+2})_{r,s}$. Each function is equipped with a Laplace-eigenvalue equation. For $F^{+(s)}_{m,n}$, we have
\begin{equation}\label{Fplus}
    (\Delta + s(s-1))F^{+(s)}_{m,n} = \dfrac{-2^{m+n+2}\mathbb{L}^{m+n-2}}{(2m-2)!(2n-2)!} \mathcal{E}_{m-1,m-1}\mathcal{E}_{n-1,n-1},
\end{equation}
where $2 \leq m \leq n$ and $s \in \{n-m+2, n-m+4, \dots , n+m-4, n+m-2  \}$. For $F^{-(s)}_{m,n}$, we have
\begin{equation*}
     (\Delta + s(s-1))F^{-(s)}_{m,n}  = \dfrac{(mn)2^{m+n+1}\mathbb{L}^{m+n-2}}{(2m-2)!(2n-2)!} \left( \mathcal{E}_{m-2,m}\mathcal{E}_{n,n-2} -\mathcal{E}_{m,m-2}\mathcal{E}_{n-2,n}  \right),
\end{equation*}
where $2 \leq m \leq n$ and $s \in \{n-m+1, n-m+3, \dots , n+m-3, n+m-1  \}$. We can see that, when $m=n=2$ and $s=2$, and $m=2$, $n=3$ and $s=3$, equation \eqref{Fplus} becomes 
\begin{align*}
    (\Delta +2)F^{+(2)}_{2,2} = -16\LL^2\mathcal{E}_{1,1}\mathcal{E}_{1,1} \qquad \text{and} \qquad 
    (\Delta + 6)F^{+(3)}_{2,3} = -\dfrac{8}{3}\LL^3\mathcal{E}_{1,1}\mathcal{E}_{2,2}.
\end{align*}
These can be viewed as simpler versions of the Laplace equations given by \eqref{c211} and \eqref{c311}, which we have already shown can be given in terms of Laplace equations for length two modular iterated integrals. Indeed, we have
\begin{align*}
    (\Delta +2)F^{+(2)}_{2,2} &= (\Delta +2)\left(  16\LL^2\mathcal{E}_{2,0}\mathcal{E}_{0,2}-4\LL (F^{(1)}_{4,4})_{1,1} \right)\!,
    \\[0.5mm]
    (\Delta + 6)F^{+(3)}_{2,3} &= (\Delta +6)\left( \dfrac{4}{3}\LL^3 \mathcal{E}_{2,0}\mathcal{E}_{1,3}-\dfrac{2}{9}\LL^2(F^{(1)}_{4,6})_{2,2}  \right)\!.
\end{align*}
Laplace equations involving the $F^{-(s)}_{m,n}$ functions have also been found to match up with Laplace equations for length two modular iterated integrals: 
\begin{align*}
    (\Delta+2)F^{-(2)}_{2,3}   &= 8\LL^3(\mathcal{E}_{0,2}\mathcal{E}_{3,1}- \mathcal{E}_{2,0}\mathcal{E}_{1,3}) \\
    &= (\Delta+2)\left( 4\LL^3\mathcal{E}_{2,0}\mathcal{E}_{1,3} - \dfrac{4}{3}\LL^3\mathcal{E}_{1,1}\mathcal{E}_{2,2} - \dfrac{4}{3}\LL(F^{(2)}_{4,6})_{1,1}  \right)\!, \\[0.5mm]
    (\Delta +12)F^{-(4)}_{2,3}  &= 8\LL^3(\mathcal{E}_{0,2}\mathcal{E}_{3,1} - \mathcal{E}_{2,0}\mathcal{E}_{1,3})  \\
    &=  (\Delta+12)  \left(\dfrac{8}{3}\LL^3(F^{(0)}_{4,6})_{3,3}  -   \dfrac{4}{3}\LL^3\mathcal{E}_{1,1}\mathcal{E}_{2,2}  -  \dfrac{8}{3}\LL^3\mathcal{E}_{2,0}\mathcal{E}_{1,3} \right)\!.
\end{align*}

\subsubsection{Example 3}

In this final example, which will be of use in Section \ref{endex3}, we set $2a=2$, $2b=6$ 
and focus on $k=0$:
\begin{align*}
&\partial (F^{(0)}_{4,8})_{8,0}= \LL\GG_4\EE_{6,0}
&& \overline{\partial} (F^{(0)}_{4,8})_{0,8}=  \LL\overline{\GG}_8\EE_{0,2} + C^{(0)}_{4,8}\LL\overline{\mathbb{G}}_{10}
\\[1mm]
&\partial (F^{(0)}_{4,8})_{7,1} - 8(F^{(0)}_{4,8})_{8,0}= \LL\GG_4\EE_{5,1} 
&&  \overline{\partial} (F^{(0)}_{4,8})_{1,7} - 8(F^{(0)}_{4,8})_{0,8} =  \LL\overline{\GG}_8\EE_{1,1}
\\[1mm]
&\partial (F^{(0)}_{4,8})_{6,2} - 7(F^{(0)}_{4,8})_{7,1}=\LL\GG_4\EE_{4,2} 
&& \overline{\partial} (F^{(0)}_{4,8})_{2,6} - 7(F^{(0)}_{4,8})_{1,7} = \LL\overline{\GG}_8\EE_{2,0}
\\[1mm]
&\partial (F^{(0)}_{4,8})_{5,3} - 6(F^{(0)}_{4,8})_{6,2}=\LL\GG_4\EE_{3,3} 
\hspace{10mm}  
&& \overline{\partial} (F^{(0)}_{4,8})_{3,5} - 6(F^{(0)}_{4,8})_{2,6} = 0
\\[1mm]
&\partial (F^{(0)}_{4,8})_{4,4} - 5(F^{(0)}_{4,8})_{5,3}=\LL\GG_4\EE_{2,4} 
&& \overline{\partial} (F^{(0)}_{4,8})_{4,4} - 5(F^{(0)}_{4,8})_{3,5} = 0
\\[1mm]
&  \partial (F^{(0)}_{4,8})_{3,5} - 4(F^{(0)}_{4,8})_{4,4}=  \LL\GG_4\EE_{1,5}
&&  \overline{\partial} (F^{(0)}_{4,8})_{5,3} - 4(F^{(0)}_{4,8})_{4,4} = 0
\\[1mm]
& \partial (F^{(0)}_{4,8})_{2,6} - 3(F^{(0)}_{4,8})_{3,5}= \LL\GG_4\EE_{0,6} 
&&  \overline{\partial} (F^{(0)}_{4,8})_{6,2} - 3(F^{(0)}_{4,8})_{5,3} = 0
\\[1mm]
& \partial (F^{(0)}_{4,8})_{1,7} - 2(F^{(0)}_{4,8})_{2,6}=0
&&  \overline{\partial} (F^{(0)}_{4,8})_{7,1} - 2(F^{(0)}_{4,8})_{6,2} = 0
\\[1mm]
&\partial (F^{(0)}_{4,8})_{0,8} - (F^{(0)}_{4,8})_{1,7}=0
&&  \overline{\partial} (F^{(0)}_{4,8})_{8,0} - (F^{(0)}_{4,8})_{7,1} = 0
\end{align*} 
\begin{align}
    &(\Delta+8)(F^{(0)}_{4,8})_{8,0} = - \LL\GG_4\EE_{5,1} \nonumber
    \\[1mm]
    &(\Delta+8)(F^{(0)}_{4,8})_{7,1} = - 2\LL\GG_4\EE_{4,2} \hspace{12mm}
    &&(\Delta+8)(F^{(0)}_{4,8})_{3,5} = - 6\LL\GG_4\EE_{0,6} \nonumber
    \\[1mm]
    &(\Delta+8)(F^{(0)}_{4,8})_{6,2} = - 3\LL\GG_4\EE_{3,3}
    &&(\Delta+8)(F^{(0)}_{4,8})_{2,6} = - \LL^2\GG_4\overline{\GG}_8 \label{ex3a}
    \\[1mm]
    &(\Delta+8)(F^{(0)}_{4,8})_{5,3} =  - 4 \LL\GG_4\EE_{2,4}
    &&(\Delta+8)(F^{(0)}_{4,8})_{1,7} = - 2 \LL\overline{\GG}_8\EE_{2,0} \label{ex3b}
    \\[1mm]
    &(\Delta+8)(F^{(0)}_{4,8})_{4,4} = - 5 \LL\GG_4\EE_{1,5}
    &&(\Delta+8)(F^{(0)}_{4,8})_{0,8} = - \LL\overline{\GG}_8\EE_{1,1}. \label{ex3c}
\end{align}
As with the previous example, even if $C^{(0)}_{4,8}$ is zero, the $\partial, \bar{\partial}$ equations cannot be solved in terms of real analytic Eisenstein series. The functions $(F^{(0)}_{4,8})_{r,s}$ are all ``new''. In fact, we get ``new'' functions for every $k$ whenever $a \neq b$. 
When $a=b$, we only get ``new'' functions for odd $k$.
Below, we will try to shed some light on why this occurs.

The form $D_{2a,2b}$ can be viewed as part of the form
\begin{align*}
    \tilde{D}_{2a,2b}:=& \ d(\mathcal{E}_{2a-2}\otimes \mathcal{E}_{2b-2}) \\ =& \ \underline{E}_{2a}\otimes\mathcal{E}_{2b-2} + \mathcal{E}_{2a-2}\otimes \underline{E}_{2b} +
    \overline{\underline{E}}_{2a}\otimes\mathcal{E}_{2b-2} + \mathcal{E}_{2a-2}\otimes \overline{\underline{E}}_{2b}.
\end{align*}

This form is closed and equivariant but, for the purposes of this paper, we do not consider this object to be very interesting. Indeed, if we undergo the same process as in Sections \ref{doub} and \ref{EVTEI}, but with $\tilde{D}_{2a,2b}$ instead of $D_{2a,2b}$, we get analogous equations to \eqref{MI3THM1} and \eqref{MI3THM2}:
\begin{align*}
       & \partial (\tilde{F}^{(k)}_{2a+2,2b+2})_{2w,0} = \binom{2a}{k} \mathbb{L}^{k+1}\mathbb{G}_{2a+2}\mathcal{E}_{2b-k,k} \\ & \hspace{32mm}  + (-1)^k\binom{2b}{k}\mathbb{L}^{k+1}\mathbb{G}_{2b+2}\mathcal{E}_{2a-k,k} + \mathbb{L}f ,  \nonumber \\[1mm]
& \partial (\tilde{F}^{(k)}_{2a+2,2b+2})_{r,s} - (r+1)(\tilde{F}^{(k)}_{2a+2,2b+2})_{r+1,s-1}  \\
& \hspace{24mm}= \binom{2a}{k}\binom{k+s}{k} \mathbb{L}^{k+1}\mathbb{G}_{2a+2}\mathcal{E}_{2b-k-s,k+s} \\ &
\hspace{32mm}  + (-1)^k \binom{2b}{k}\binom{k+s}{k} \mathbb{L}^{k+1}\mathbb{G}_{2b+2}\mathcal{E}_{2a-k-s,k+s}  , \hspace{2mm}  \text{  if } s \geq 1, \nonumber
\end{align*}
where $f$ is a cusp form of weight $2w+2$.
Then, when $2a=2b=2$ and $k=0$, we have:
\begin{align*}
&\partial (\tilde{F}^{(0)}_{4,4})_{4,0}=\mathbb{L}\mathbb{G}_4 \mathcal{E}_{2,0} + \mathbb{L}\mathbb{G}_4 \mathcal{E}_{2,0}
\\[0.5mm]
&\partial (\tilde{F}^{(0)}_{4,4})_{3,1} - 4(\tilde{F}^{(0)}_{4,4})_{4,0}=\mathbb{L}\mathbb{G}_4 \mathcal{E}_{1,1} + \mathbb{L}\mathbb{G}_4 \mathcal{E}_{1,1}
\hspace{3mm}  
&& \partial (\tilde{F}^{(0)}_{4,4})_{1,3} - 2(\tilde{F}^{(0)}_{4,4})_{2,2}=0 + 0
\\[0.5mm]
&\partial (\tilde{F}^{(0)}_{4,4})_{2,2} - 3(\tilde{F}^{(0)}_{4,4})_{3,1}=\mathbb{L}\mathbb{G}_4 \mathcal{E}_{0,2} + \mathbb{L}\mathbb{G}_4 \mathcal{E}_{0,2}
\hspace{3mm}  
&& \partial (\tilde{F}^{(0)}_{4,4})_{0,4} - (\tilde{F}^{(0)}_{4,4})_{1,3}=0 + 0.
\end{align*} 

The reason why we keep the RHS of the equations as two separate summands will soon become clear.
When $k=1$, we have the equations:
\begin{align*}
&\partial (\tilde{F}^{(1)}_{4,4})_{2,0}=2\mathbb{L}^2\mathbb{G}_4 \mathcal{E}_{1,1}  - 2\mathbb{L}^2\mathbb{G}_4 \mathcal{E}_{1,1}
\hspace{22mm} \partial (\tilde{F}^{(1)}_{4,4})_{0,2} - (\tilde{F}^{(1)}_{4,4})_{1,1}= 0 - 0
\\[0.5mm]
&\partial (\tilde{F}^{(1)}_{4,4})_{1,1} - 2(\tilde{F}^{(1)}_{4,4})_{2,0}=4\mathbb{L}^2\mathbb{G}_4 \mathcal{E}_{0,2}  - 4\mathbb{L}^2\mathbb{G}_4 \mathcal{E}_{0,2}.
\end{align*} 

Next, when $k=2$ we have:
\begin{align*}
    &\partial (\tilde{F}^{(2)}_{4,4})_{0,0} = \mathbb{L}^3 \mathbb{G}_4 \mathcal{E}_{0,2} + \mathbb{L}^3 \mathbb{G}_4 \mathcal{E}_{0,2}.
    \end{align*}
    
We finish by looking at when $2a=2$, $2b=4$ and $k=0$:
\begin{align}
&\partial (\tilde{F}^{(0)}_{4,6})_{6,0}=  \mathbb{L}\mathbb{G}_4\mathcal{E}_{4,0} + \mathbb{L}\mathbb{G}_6 \mathcal{E}_{2,0} 
&& \label{tildeF}
\\
&\partial (\tilde{F}^{(0)}_{4,6})_{5,1} - 6(\tilde{F}^{(0)}_{4,6})_{6,0}= \mathbb{L}\mathbb{G}_4\mathcal{E}_{3,1} + \mathbb{L}\mathbb{G}_6 \mathcal{E}_{1,1}   \nonumber
\\
&\partial (\tilde{F}^{(0)}_{4,6})_{4,2} - 5(\tilde{F}^{(0)}_{4,6})_{5,1}= \mathbb{L}\mathbb{G}_4\mathcal{E}_{2,2} + \mathbb{L}\mathbb{G}_6 \mathcal{E}_{0,2}
 \nonumber
\\
&\partial (\tilde{F}^{(0)}_{4,6})_{3,3} - 4(\tilde{F}^{(0)}_{4,6})_{4,2}=\mathbb{L}\mathbb{G}_4\mathcal{E}_{1,3}+0
&&   \partial (\tilde{F}^{(0)}_{4,6})_{1,5} - 2(\tilde{F}^{(0)}_{4,6})_{2,4}=0 + 0 \nonumber
\\
&  \partial (\tilde{F}^{(0)}_{4,6})_{2,4} - 3(\tilde{F}^{(0)}_{4,6})_{3,3}= \mathbb{L}\mathbb{G}_4\mathcal{E}_{0,4}+0 \nonumber
&&  \partial (\tilde{F}^{(0)}_{4,6})_{0,6} - (\tilde{F}^{(0)}_{4,6})_{1,5}=0 + 0. \nonumber
\end{align} 
    
All of the above equations, however, can be solved in terms of real analytic Eisenstein series. For example, $(\tilde{F}^{(0)}_{4,6})_{6,0} = \mathcal{E}_{4,0}\mathcal{E}_{2,0}$ and $(\tilde{F}^{(0)}_{4,6})_{5,1} = \mathcal{E}_{4,0}\mathcal{E}_{1,1}+ \mathcal{E}_{2,0}\mathcal{E}_{3,1}$. This is why we do not consider these functions, and hence $\tilde{D}_{2a,2b}$, to be very interesting.

Now, we define a new function $ (F'^{(0)}_{4,6})_{6,0}$ by removing one of the summands from the RHS of equation \eqref{tildeF}, say $\mathbb{L}\mathbb{G}_6 \mathcal{E}_{2,0}$, to give $\partial (F'^{(0)}_{4,6})_{6,0} :=\mathbb{L}\mathbb{G}_4\mathcal{E}_{4,0}$. This is no longer solvable in terms of real analytic Eisenstein series. However, the equation for $\partial ({F'}^{(0)}_{4,6})_{6,0}$ is exactly the equation for $\partial ({F}^{(0)}_{4,6})_{6,0}$, from Example 2. This same argument holds for any of the  $(\tilde{F}^{(k)}_{2a+2,2b+2})_{r,s}$. 

Therefore, the equations defining the $\partial ({F}^{(k)}_{2a+2,2b+2})_{r,s}$ can be thought, in an oversimplified way, as coming from the equations for $\partial (\tilde{F}^{(k)}_{2a+2,2b+2})_{r,s}$, but with certain summands removed. It is then the removal of these summands that gives rise to the ``new'' functions that have been found. Or, in other words, we construct $D_{2a,2b}$ by
\begin{equation*}
   D_{2a,2b} := \tilde{D}_{2a,2b} \, - \, \mathcal{E}_{2a-2} \otimes \underline{E}_{2b} \, - \, \underline{\overline{E}}_{2a} \otimes \mathcal{E}_{2b-2}.
\end{equation*}
We are also very fortunate that removing terms from $\tilde{D}_{2a,2b}$ to give $D_{2a,2b}$ does not also remove the closure. 

We finish by noting that, whenever $2a=2b$ and $k$ is even, we have 
\begin{equation*}
    \partial(F^{(k)}_{2a+2,2b+2})_{r,s} =\dfrac{1}{2} \partial(\tilde{F}^{(k)}_{2a+2,2b+2})_{r,s},
\end{equation*}
as can be seen in the examples above. This equality sheds some light on why the removal of summands does not give ``new'' functions under these specific conditions.

\section{The Length Three Case}\label{length3}
\subsection{Triple Integrals}

The main obstacle with extending the theory from the previous section to $\mathcal{MI}_3$ is finding a closed modular equivariant object, say $D_{2a,2b,2c}$, which also leads to what we will consider to be ``new" functions in $\mathcal{MI}_3$. (An exact definition of a ``new'' length three function will be given after Theorem \ref{3MI3THM}.) In the previous section we introduced the form $\tilde{D}_{2a,2b}$, we have a similar definition for the length three case:
 \begin{align*}
 &\tilde{D}_{2a,2b,2c}   := \  d(\mathcal{E}_{2a-2} \otimes \mathcal{E}_{2b-2} \otimes \mathcal{E}_{2c-2}) \\ & \quad =   \underline{E}_{2a}\otimes\mathcal{E}_{2b-2} \otimes \mathcal{E}_{2c-2} \ + \ \mathcal{E}_{2a-2}\otimes\underline{E}_{2b}\otimes\mathcal{E}_{2c-2} \ + \ \mathcal{E}_{2a-2} \otimes \mathcal{E}_{2b-2} \otimes \underline{E}_{2c}  \\
& \quad \quad +\ \mathcal{E}_{2a-2} \otimes \mathcal{E}_{2b-2} \otimes \underline{\overline{E}}_{2c} \ + \ \mathcal{E}_{2a-2}\otimes\underline{\overline{E}}_{2b}\otimes\mathcal{E}_{2c-2} \ + \ \underline{\overline{E}}_{2a}\otimes\mathcal{E}_{2b-2} \otimes \mathcal{E}_{2c-2}.
 \end{align*}
This may seem like a suitable candidate as it is both closed and modular equivariant. However, as with $\tilde{D}_{2a,2b}$, when following the same steps from Sections \ref{doub} and \ref{EVTEI}, but with $\tilde{D}_{2a,2b,2c}$ instead of $D_{2a,2b}$, the resulting theorem does not lead to any ``new" functions. 

In the length two case, we saw how we could define $D_{2a,2b}$ by removing summands from $\tilde{D}_{2a,2b}$; this resulting form was also closed. However, in the length three case, removing any combination of summands from $\tilde{D}_{2a,2b,2c}$ will result in a form that is not closed.  (In fact, this result generalises for any length $n>2$. We can define a form $\tilde{D}_{a_1, a_2, \dots , a_n} := d(\mathcal{E}_{2a_1-2}\otimes \mathcal{E}_{2a_2-2} \otimes \cdots \otimes \mathcal{E}_{2a_n-2})$. As far as the author is aware, removing any combination of summands from $\tilde{D}_{a_1, a_2, \dots , a_n}$  results in a form that is not closed. It appears, then, that the length two case is unique in the ability to remove summands from $\tilde{D}_{2a,2b}$ whilst keeping the property of closure.) 

Therefore, we have the problem of finding an object that is modular equivariant, closed and leads to ``new'' functions.
In order to solve this problem, we use the functions produced by Theorem \ref{MI3THM} directly in our definition of $D_{2a,2b,2c}$. 

\begin{definition}\label{defabc}
For $a,b,c \geq 1$, we define the forms $D_{2a+2,2b+2,2c+2}$ by
\begin{align*}
    D_{2a+2,2b+2,2c+2}(z) &: \mathcal{H} \to \left(V_{2a+2b+2c}  \right) \otimes \left(\mathbb{C} \, dz + \mathbb{C} \,d\bar z \right) \\
     D_{2a+2,2b+2,2c+2}(z) &:= E_{2a+2}(z) \cdot 
    F^{(0)}_{2b+2,2c+2}(z) \ dz + F^{(0)}_{2a+2,2b+2}(z) \cdot \overline{E}_{2c+2}(z) 
    \  d\bar{z}
    \\[1mm]
     & \qquad + 2\pi i  \mathcal{E}_{2c}(z)f^{(0)}_{2a+2,2b+2}(z)(X-z Y)^{2a+2b}  dz \\[1mm]
    & \qquad\qquad -2\pi i  \mathcal{E}_{2a}(z)\bar{g}^{(0)}_{2b+2,2c+2}(z)(X-\bar{z}Y)^{2b+2c} d\bar{z}.
\end{align*}
\end{definition}

We will see in Lemma \ref{closedlem} that these last two summands are included to deal with the modular forms that appear in equations \eqref{dFdz} and \eqref{dFdz2}. 

We know that $E_{2a+2}(z)dz$, $\overline{E}_{2c+2}(z)d\bar{z}$, $ \mathcal{E}_{2a}$ and $\mathcal{E}_{2c}$ are modular equivariant and, by Section \ref{EVTEI}, $F^{(0)}_{2a+2,2b+2}$ and $F^{(0)}_{2b+2,2c+2}$ are also equivariant. Now, as the modular form $f^{(0)}_{2a+2,2b+2}$ (resp. $g^{(0)}_{2b+2,2c+2}$) is of weight $2a+2b+2$ (resp. $2b+2c+2$), we have the modular equivariance of $D_{2a+2,2b+2,2c+2}$:
\begin{equation*}
        D_{2a+2,2b+2,2c+2}(\gamma z)|_{\gamma} = D_{2a+2,2b+2,2c+2}(z), \qquad \forall \ \gamma\in \Gamma_1.
\end{equation*}

\begin{lemma}\label{closedlem}
The forms $D_{2a+2,2b+2,2c+2}$ are closed:
\begin{equation*}
    d D_{2a+2,2b+2,2c+2} = 0.
\end{equation*}
\end{lemma}
\begin{proof}
Throughout this proof we use the abbreviations $f:= {f}^{(0)}_{2a+2,2b+2}$ and $\bar{g} := \bar{g}^{(0)}_{2b+2,2c+2}$.
For a form $\sigma(z) := u_1(z) \, dz + u_2(z) \, d\bar{z}$, the exterior derivative is given by
\begin{equation*}
    d\sigma = \left( \dfrac{\partial u_1}{\partial \bar{z}}  - \dfrac{\partial u_2}{\partial z} \right) \  d\bar{z} \wedge  dz. 
\end{equation*}
For $\sigma(z) =  D_{2a+2,2b+2,2c+2}(z)$, we have
\begin{align}
     \dfrac{\partial u_1}{\partial \bar{z}}  - & \dfrac{\partial u_2}{\partial z} = E_{2a+2}(z) 
    \dfrac{\partial}{\partial \bar{z}}F^{(0)}_{2b+2,2c+2}(z)  +\pi i \overline{E}_{2c+2}(z) f(z)(X-z Y)^{2a+2b} \nonumber  
    \\ &  - \dfrac{\partial}{\partial z}F^{(0)}_{2a+2,2b+2}(z) \cdot \overline{E}_{2c+2}(z) +\pi i  E_{2a+2}(z) \overline{g}(z) (X-\bar{z}Y)^{2b+2c} \label{closed1}.
\end{align}
Now, using equation \eqref{dFdz} for $k=0$, we deduce that
\begin{equation}\label{closea}
   \dfrac{\partial }{\partial z} F^{(0)}_{2a+2,2b+2}(z)=  \dfrac{1}{2}
   E_{2a+2}(z) \mathcal{E}_{2b}(z) + \pi i f(z)(X-z Y)^{2a+2b}
\end{equation}
and, similarly, we have
\begin{equation}\label{closeb}
    \dfrac{\partial}{\partial \bar{z}}F^{(0)}_{2b+2,2c+2}(z) =  \dfrac{1}{2} \mathcal{E}_{2b}(z)  \overline{E}_{2c+2}(z) - \pi i \bar{g}(z)(X-\bar{z} Y)^{2b+2c}. 
\end{equation}
Inserting these two equations back into \eqref{closed1} gives the required result.
\end{proof}

The fact that the first summand on the RHS of equations \eqref{closea} and \eqref{closeb} contains a matching $\mathcal{E}_{2b}$ is a key part of the above proof. These equations also demonstrate why generalising Definition \ref{defabc} for $k>0$ is not as straightforward as the length two case. We will take a look at what happens to these equations when $k=1$. The first summand on the RHS of \eqref{closea}, which will correspond to the first summand on the RHS of \eqref{dFdz} when $k=1$, becomes
\begin{equation}\label{closec}
    a\mathbb{L} E_{2a+2}(z) \! \sum_{r+s=2b} s \, \mathcal{E}_{r,s}(X-zY)^{r-1}(X-\bar{z}Y)^{s-1},
\end{equation}
whilst the first summand on the RHS of \eqref{closeb} becomes
\begin{equation}\label{closee}
   \sum_{r+s=2b} r \, \mathcal{E}_{r,s} (X-zY)^{r-1}(X-\bar{z}Y)^{s-1} \cdot c\mathbb{L}\overline{E}_{2c+2}(z).
\end{equation}
Due to the appearance of $s$ inside one summand and $r$ inside the other, we conclude that we are not able to easily edit the proof of Lemma \ref{closedlem} to hold for $k=1$. The author, however, does not believe that a version of Definition \ref{defabc} for $k=1$ (or higher) is out of reach, but it will require extra terms, not yet found, to deal with the difference between equations \eqref{closec} and \eqref{closee}.  

Lemma \ref{closedlem} leads to the definition of the following function:
\begin{equation*}
    K_{2a+2,2b+2,2c+2}(z) = -\dfrac{1}{2} \int_{z}^{\reg} D_{2a+2,2b+2,2c+2}.
\end{equation*}
This integral depends only on the homotopy class of the chosen path and is regularised as in Section \ref{doub}.
This function also satisfies the following differential equations
\begin{align}
    \dfrac{\partial}{\partial z} K_{2a+2,2b+2,2c+2}(z) &= \pi i  \mathbb{G}_{2a+2}(z)(X-zY)^{2a} \cdot F^{(0)}_{2b+2,2c+2}(z)  \label{3dzk1}\\
    & \hspace{14mm} +\pi i  \mathcal{E}_{2c}(z)f^{(0)}_{2a+2,2b+2}(z)(X-z Y)^{2a+2b}
, \nonumber
    \\[1mm]
     \dfrac{\partial}{\partial \bar{z}} K_{2a+2,2b+2,2c+2}(z) &=  F^{(0)}_{2a+2,2b+2} (z)\cdot \overline{\pi i  \mathbb{G}_{2c+2}}(z)(X-\bar{z}Y)^{2c} \label{3dzbar} \\
     &\hspace{14mm} - \pi i  \mathcal{E}_{2a}(z)\bar{g}^{(0)}_{2b+2,2c+2}(z)(X-\bar{z}Y)^{2b+2c}. \nonumber
\end{align}
If we set
\begin{equation*}
    J(z) :=\mathbb{G}_{2a+2}(z)(X-zY)^{2a} F^{(0)}_{2b+2,2c+2}(z)+  \mathcal{E}_{2c}(z)f^{(0)}_{2a+2,2b+2}(z)(X-z Y)^{2a+2b},
\end{equation*}
then, by equations \eqref{e2a} and \eqref{Fkab}, 
\begin{align*}
    J(z) =& \ \mathbb{G}_{2a+2}(z)(X-zY)^{2a}\!\! \sum_{r+s=2b+2c} \!\! \left(F^{(0)}_{2b+2,2c+2} \right)_{r,s}(X-zY)^r(X-\bar{z}Y)^s \\
    & \hspace{15mm} + \sum_{r+s=2c}\mathcal{E}_{r,s}(X-zY)^r(X-\bar{z}Y)^s  f^{(0)}_{2a+2,2b+2}(X-zY)^{2a+2b} 
    \\
    = & \ \mathbb{G}_{2a+2}(z)\!\! \sum_{r+s=2b+2c} \!\! \left(F^{(0)}_{2b+2,2c+2} \right)_{r,s}(X-zY)^{r+2a}(X-\bar{z}Y)^s 
     \\ & \hspace{30mm} + f^{(0)}_{2a+2,2b+2} \sum_{r+s=2c}\mathcal{E}_{r,s}(X-zY)^{r+2a+2b}(X-\bar{z}Y)^s  .
\end{align*}
Using the change of variables $r \mapsto r-2a$ in the first sum and $r \mapsto r-2a-2b$ in the second sum, we conclude that
\begin{align*}
J(z)     = & \ \mathbb{G}_{2a+2}(z) \sum_{\mathclap{\substack{r+s= \\ 2a+2b+2c}}}  \left(F^{(0)}_{2b+2,2c+2} \right)_{r-2a,s}(X-zY)^r(X-\bar{z}Y)^s 
     \\ & \hspace{38mm} + f^{(0)}_{2a+2,2b+2} \sum_{\mathclap{\substack{r+s= \\
    2a+2b+2c}}}\mathcal{E}_{r-2a-2b,s}(X-zY)^r(X-\bar{z}Y)^s .
\end{align*}
Therefore, equation \eqref{3dzk1} becomes
\begin{equation}\label{3dzk3}
    \dfrac{\partial}{\partial z} K_{2a+2,2b+2,2c+2}(z)= \pi i \, J(z)
\end{equation}
with
\begin{equation}\label{3dzk4}
    J_{r,s} = \mathbb{G}_{2a+2} \left(F^{(0)}_{2b+2,2c+2}\right)_{r-2a,s} + f^{(0)}_{2a+2,2b+2} \mathcal{E}_{r-2a-2b,s} .
\end{equation}

\subsection{Equivariant Versions of Triple Integrals}

Using equations \eqref{3dzk1} and \eqref{3dzbar} with the same reasoning as used in Section \ref{EVTEI}, we have that, for all $\gamma \in \Gamma_1$,
\begin{align*}
    K_{2a+2,2b+2,2c+2}(\gamma z)|_{\gamma}-K_{2a+2,2b+2,2c+2}(z) =C_{\gamma},
\end{align*}
for some constant $C_{\gamma}$ (does not depend on $z$) which also defines a coboundary and hence a cocycle:
\begin{equation*}
    \gamma \mapsto C_{\gamma} \in Z^1(\Gamma_1,V_{2a+2b+2c} \otimes \mathbb{C}).
\end{equation*}
As before, we then define a function $M_{2a+2,2b+2,2c+2}$ such that we have modular equivariance and 
\begin{equation*}
    M_{2a+2,2b+2,2c+2} = K_{2a+2,2b+2,2c+2} - c - \dfrac{1}{2} \int^{\reg}_{z} (\underline{h} + \overline{\underline{k}}),
\end{equation*}
where $h$ (resp. $k$) is a standard modular form (resp. modular cusp form) of weight $2a+2b+2c+2$. The equation above uniquely determines $c$, $h$, $k$ and $M_{2a+2,2b+2,2c+2}$.

\begin{theorem}\label{3MI3THM}
Let $a,b,c \geq 1$ and set $w=a+b+c$. There exists a family of elements $(G_{2a+2,2b+2,2c+2})_{r,s} \in \mathcal{MI}_3 \cap \mathcal{M}_{r,s} $ of total modular weight $2w = r+s$, with $r,s \geq 0$, which satisfy the following:
\begin{align}
     a)
     \ & \partial (G_{2a+2,2b+2,2c+2})_{2w,0} =  \mathbb{L}\mathbb{G}_{2a+2} \left(F^{(0)}_{2b+2,2c+2}\right)_{2b+2c,0} \label{3MI3THM1} \\ & \hspace{60mm} + \mathbb{L}f^{(0)}_{2a+2,2b+2} \mathcal{E}_{2c,0}   + \mathbb{L}h, \nonumber \\[1mm]
    \ & \partial (G_{2a+2,2b+2,2c+2})_{r,s} - (r+1)(G_{2a+2,2b+2,2c+2})_{r+1,s-1} \label{3MI3THM2}\\ & \hspace{10mm} = \mathbb{L}\mathbb{G}_{2a+2} \left(F^{(0)}_{2b+2,2c+2}\right)_{r-2a,s} \!+ \mathbb{L}f^{(0)}_{2a+2,2b+2} \mathcal{E}_{r-2a-2b,s}  , \hspace{4mm}  \tag*{if  $s \geq 1$,} \nonumber
\end{align}
where $h$ is a cusp form of weight $2w+2$.
\begin{align}
      b)  \ & \overline{\partial} (G_{2a+2,2b+2,2c+2})_{0,2w} =  \mathbb{L}\overline{\mathbb{G}}_{2c+2}\left(F^{(0)}_{2a+2,2b+2}\right)_{0,2a+2b} \label{3thmk11} \\ & \hspace{60mm} + \mathbb{L}\bar{g}^{(0)}_{2b+2,2c+2} \mathcal{E}_{0,2a}
+ \mathbb{L}\overline{k}, \nonumber  \\[1mm]
   & \overline{\partial} (G_{2a+2,2b+2,2c+2})_{r,s} - (s+1)(G_{2a+2,2b+2,2c+2})_{r-1,s+1} \label{3thmk12}\\ 
   &\hspace{10mm} =\mathbb{L}\overline{\mathbb{G}}_{2c+2}\left(F^{(0)}_{2a+2,2b+2}\right)_{r,s-2c} \!+ \mathbb{L}\bar{g}^{(0)}_{2b+2,2c+2} \mathcal{E}_{r, s-2b-2c},   \hspace{4mm} \tag*{   if  $r \geq 1$,} \nonumber
\end{align}
where $k$ is a modular form of weight $2w+2$.
\end{theorem}
\begin{proof}
Throughout this proof we use the abbreviation $M_{abc}\!:=\!M_{2a+2,2b+2,2c+2}$. The function $M_{2a+2,2b+2,2c+2}$ is, by definition, equivariant. Using Lemma \ref{7.1}, we obtain its modular components $(M_{2a+2,2b+2,2c+2})_{r,s}$, which are elements of $\mathcal{M}_{r,s}$. Now, looking at $\dfrac{\partial}{\partial z}M_{2a+2,2b+2,2c+2}$, we have
\begin{align}
    \dfrac{\partial}{\partial z}M_{abc}
    &= \dfrac{\partial}{\partial z}K_{2a+2,2b+2,2c+2} - 
    \dfrac{\partial}{\partial z}\int^{\reg}_{z}  \pi i h(\tau)(X-\tau Y)^{2w}d\tau \\
   & \hspace{42mm}-\dfrac{\partial}{\partial z}\int^{\reg}_{z}  \pi i \overline{k}(\tau)(X-\bar{\tau}Y)^{2w}d\bar{\tau}  \nonumber
   \\
   &= \dfrac{\partial}{\partial z}K_{2a+2,2b+2,2c+2} + \pi i  h(z)(X-zY)^{2w}. \nonumber 
   \end{align}
We write this as
\begin{equation}\label{3bigproof1}
    \dfrac{\partial}{\partial z}M_{2a+2,2b+2,2c+2} = \pi i \, H,
\end{equation}
where, using equations \eqref{3dzk3} and \eqref{3dzk4}, 
\begin{equation*}
    H_{r,s}=  \mathbb{G}_{2a+2} \left(F^{(0)}_{2b+2,2c+2}\right)_{r-2a,s} + f^{(0)}_{2a+2,2b+2} \mathcal{E}_{r-2a-2b,s} + A_{r,s}
\end{equation*}
and $A_{r,s} = 0$ except for $A_{2w,0}=h$.
Lemma \ref{dF=A} tells us that equation \eqref{3bigproof1} is equal to the following system of equations, for all $r+s=2w$ and $r,s \geq 0$,
\begin{align*}
    &\partial (M_{abc})_{2w,0} \! =  \mathbb{L} H_{2w,0}\!=  \mathbb{L} \mathbb{G}_{2a+2}\!\left(\!F^{(0)}_{2b+2,2c+2}\right)_{r-2a,s} \!\! +  \mathbb{L}f^{(0)}_{2a+2,2b+2} \mathcal{E}_{r-2a-2b,s} +\mathbb{L}h ,\\[1mm]
    &\partial (M_{abc})_{r,s} - (r+1)(M_{abc})_{r+1,s-1} =\mathbb{L}H_{r,s} \\ & \hspace{15mm}= \mathbb{L} \mathbb{G}_{2a+2}\left(F^{(0)}_{2b+2,2c+2}\right)_{r-2a,s}  + \mathbb{L}f^{(0)}_{2a+2,2b+2} \mathcal{E}_{r-2a-2b,s} , \hspace{14mm} \text{if } s \geq 1.
\end{align*}
Modifying $M_{2a+2,2b+2,2c+2}$ by a suitable multiple of $\mathcal{E}_{2w}$ allows us to assume that $h$ is a cusp form. We note that by doing so, we can no longer guarantee $k$ is a cusp form. This gives equations \eqref{3MI3THM1} and \eqref{3MI3THM2}.
 Now $\mathbb{G}_{2a+2}$ and $f^{(0)}_{2a+2,2b+2} \in M[\mathbb{L}]$ and each $\left(F^{(0)}_{2b+2,2c+2}\right)_{r-2a,s}$ and $\mathcal{E}_{r-2a-2b,s} \in \mathcal{MI}_2$, therefore
the RHS of equations \eqref{3MI3THM1} and \eqref{3MI3THM2} exist in $M[\mathbb{L}] \times \mathcal{MI}_2$.

A similar method can be used to obtain equations \eqref{3thmk11} and \eqref{3thmk12} and to show that the RHS of these equations exist in $\overline{M}[\mathbb{L}] \times \mathcal{MI}_2$.

Then, by Lemma \ref{thmFrs}, we conclude that the $(G_{2a+2,2b+2,2c+2})_{r,s} \in \mathcal{MI}_3$.
\end{proof}

As before, since there are no cusp forms of weight $\leq 10$, equation \eqref{3MI3THM1}, for $2a+2b+2c \leq 8$, does not contain any cusp forms (the cusp form $f^{(0)}_{2a+2,2b+2}$ vanishes since the above inequality implies that $2a+2b\leq8$). Furthermore, since there are also no cusp forms of weight $14$, we have no cusp forms appearing in equation \eqref{3MI3THM1} if $2a+2b+2c=12$ with $2a+2b \leq 8$. Finally, when $2a+2b+2c \leq 8$ or $2a+2b+2c=12$, we can write $\overline{k}$ as $C_{2a+2,2b+2,2c+2} \overline{\mathbb{G}}_{2a+2b+2c+2}$, for some constant $C_{2a+2,2b+2,2c+2}$. We also recall that, for $2b+2c\leq 8$, we have $\bar{g}^{(0)}_{2b+2,2c+2}=C^{(0)}_{2b+2,2c+2}\overline{\mathbb{G}}_{2b+2c+2}$.

Let $2a+2b+2c \leq 8$, or $2a+2b+2c=12$ with $2a+2b\leq8$ and $2b+2c\leq8$. Then, we define the function $(G_{2a+2,2b+2,2c+2})_{r,s}$ to be ``new'' if its associated $\partial, \bar{\partial}$ equations cannot be solved in terms of real analytic Eisenstein series when $C_{2a+2,2b+2,2c+2}$ and $C^{(0)}_{2b+2,2c+2}$ are set to zero.

\subsection{Length Three Examples}
\subsubsection{Example 1}
The first, and simplest, example we look at is when $2a=2b=2c=2$.
Using Theorem \ref{3MI3THM} and the fact, from \cite{brown1}, that $\overline{g}^{(0)}_{4,4}$ vanishes, we have:
\begingroup
\allowdisplaybreaks
\begin{align*}
&\partial (G_{4,4,4})_{6,0}=\mathbb{L}\mathbb{G}_4 (F^{(0)}_{4,4})_{4,0} \hspace{26mm}
&& \overline{\partial}(G_{4,4,4})_{0,6} = \mathbb{L}\overline{\mathbb{G}}_4 (F^{(0)}_{4,4})_{0,4} 
\\ &
&& \hspace{30mm} + C_{4,4,4}\LL\overline{\mathbb{G}}_8
\\[1mm]
&\partial (G_{4,4,4})_{5,1} - 6(G_{4,4,4})_{6,0} \quad
&& \overline{\partial}(G_{4,4,4})_{1,5} - 6(G_{4,4,4})_{0,6} 
\\
& \hspace{25mm}=\mathbb{L}\mathbb{G}_4 (F^{(0)}_{4,4})_{3,1}
&& \hspace{25mm} = \mathbb{L}\overline{\mathbb{G}}_4 (F^{(0)}_{4,4})_{1,3}
\\[1mm]
&\partial (G_{4,4,4})_{4,2} - 5(G_{4,4,4})_{5,1} 
&& \overline{\partial}(G_{4,4,4})_{2,4} - 5(G_{4,4,4})_{1,5} 
\\&\hspace{25mm} =\mathbb{L}\mathbb{G}_4 (F^{(0)}_{4,4})_{2,2}
&&\hspace{25mm} = \mathbb{L}\overline{\mathbb{G}}_4 (F^{(0)}_{4,4})_{2,2}
\\[1mm]
&\partial (G_{4,4,4})_{3,3} - 4(G_{4,4,4})_{4,2}
&& \overline{\partial}(G_{4,4,4})_{3,3} - 4(G_{4,4,4})_{2,4} 
\\ &\hspace{25mm} =\mathbb{L}\mathbb{G}_4 (F^{(0)}_{4,4})_{1,3}
&&\hspace{25mm} = \mathbb{L}\overline{\mathbb{G}}_4 (F^{(0)}_{4,4})_{3,1}
\\[1mm]
&\partial (G_{4,4,4})_{2,4} - 3(G_{4,4,4})_{3,3}
&& \overline{\partial}(G_{4,4,4})_{4,2} - 3(G_{4,4,4})_{3,3} 
\\
&\hspace{25mm} =\mathbb{L}\mathbb{G}_4 (F^{(0)}_{4,4})_{0,4}
&&\hspace{25mm} = \mathbb{L}\overline{\mathbb{G}}_4 (F^{(0)}_{4,4})_{4,0}
\\[1mm]
&\partial (G_{4,4,4})_{1,5} - 2(G_{4,4,4})_{2,4}=0 
&& \overline{\partial}(G_{4,4,4})_{5,1} - 2(G_{4,4,4})_{4,2} = 0
\\[1mm]
&\partial (G_{4,4,4})_{0,6} - (G_{4,4,4})_{1,5}=0
&& \overline{\partial}(G_{4,4,4})_{6,0} - (G_{4,4,4})_{5,1} = 0,
\end{align*} 
\endgroup
where $C_{4,4,4}$ is some constant.
Using the equations from Section \ref{ex1} gives:
\begingroup
\begin{align*}
&\partial (G_{4,4,4})_{6,0}=\frac{1}{2}\mathbb{L}\mathbb{G}_4\mathcal{E}_{2,0}^2 \hspace{26mm}
&& \overline{\partial} (G_{4,4,4})_{0,6} = \frac{1}{2}\mathbb{L}\overline{\mathbb{G}}_4\mathcal{E}_{0,2}^2
+ C_{4,4,4}\LL\overline{\mathbb{G}}_8
\\[1mm]
&\partial (G_{4,4,4})_{5,1} - 6(G_{4,4,4})_{6,0}
&& \overline{\partial} (G_{4,4,4})_{1,5} - 6(G_{4,4,4})_{0,6} 
\\ & \hspace{25mm}=\mathbb{L}\mathbb{G}_4\mathcal{E}_{2,0}\mathcal{E}_{1,1} 
&& \hspace{25mm} =  \mathbb{L}\overline{\mathbb{G}}_4\mathcal{E}_{0,2}\mathcal{E}_{1,1}
\\[1mm]
&\partial (G_{4,4,4})_{4,2} - 5(G_{4,4,4})_{5,1}
&& \overline{\partial}(G_{4,4,4})_{2,4} - 5(G_{4,4,4})_{1,5} 
\\ & \hspace{6mm} =\mathbb{L}\mathbb{G}_4 \mathcal{E}_{2,0}\mathcal{E}_{0,2} + \frac{1}{2}\mathbb{L}\mathbb{G}_4\mathcal{E}_{1,1}^2
&& \hspace{6mm} =\mathbb{L}\overline{\mathbb{G}}_4 \mathcal{E}_{0,2}\mathcal{E}_{2,0} + \frac{1}{2}\mathbb{L}\overline{\mathbb{G}}_4\mathcal{E}_{1,1}^2
\\[1mm]
&\partial (G_{4,4,4})_{3,3} - 4(G_{4,4,4})_{4,2}
&& \overline{\partial}(G_{4,4,4})_{3,3} - 4(G_{4,4,4})_{2,4} 
\\ & \hspace{25mm} =\mathbb{L}\mathbb{G}_4\mathcal{E}_{1,1}\mathcal{E}_{0,2}
&& \hspace{25mm} = \mathbb{L}\overline{\mathbb{G}}_4\mathcal{E}_{1,1}\mathcal{E}_{2,0}
\\[1mm]
& \partial (G_{4,4,4})_{2,4} - 3(G_{4,4,4})_{3,3}
&& \overline{\partial}(G_{4,4,4})_{4,2} - 3(G_{4,4,4})_{3,3} 
\\& \hspace{25mm} = \frac{1}{2}\mathbb{LG}_4 \mathcal{E}^2_{0,2}
&& \hspace{25mm} = \frac{1}{2}\mathbb{L}\overline{\mathbb{G}}_4 \mathcal{E}^2_{2,0}
\\[1mm]
& \partial (G_{4,4,4})_{1,5} - 2(G_{4,4,4})_{2,4}=0
&& \overline{\partial}(G_{4,4,4})_{5,1} - 2(G_{4,4,4})_{4,2} = 0 
\\[1mm]
& \partial (G_{4,4,4})_{0,6} - (G_{4,4,4})_{1,5}=0
&& \overline{\partial}(G_{4,4,4})_{6,0} - (G_{4,4,4})_{5,1} = 0.
\end{align*} 
\endgroup
If $C_{4,4,4}$ is equal to zero, then these equations can be solved in terms of real analytic Eisenstein series:
\begin{align*}
&(G_{4,4,4})_{6,0}= \frac{1}{6}\mathcal{E}_{2,0}\mathcal{E}_{2,0}\mathcal{E}_{2,0}
\\
&(G_{4,4,4})_{5,1}= \frac{1}{2}\mathcal{E}_{2,0}\mathcal{E}_{2,0}\mathcal{E}_{1,1}
\\
&(G_{4,4,4})_{4,2}= \frac{1}{2}\mathcal{E}_{2,0}\mathcal{E}_{2,0}\mathcal{E}_{0,2} +\frac{1}{2}\mathcal{E}_{2,0}\mathcal{E}_{1,1}\mathcal{E}_{1,1} \quad 
\\
&(G_{4,4,4})_{3,3}= \mathcal{E}_{2,0}\mathcal{E}_{0,2}\mathcal{E}_{1,1} + \frac{1}{6}\mathcal{E}_{1,1}\mathcal{E}_{1,1}\mathcal{E}_{1,1} \hspace{20mm}
&& (G_{4,4,4})_{1,5} = \frac{1}{2}\mathcal{E}_{1,1}\mathcal{E}_{0,2}\mathcal{E}_{0,2}
\\
& (G_{4,4,4})_{2,4}= \frac{1}{2}\mathcal{E}_{2,0}\mathcal{E}_{0,2}\mathcal{E}_{0,2} + \frac{1}{2}\mathcal{E}_{1,1}\mathcal{E}_{0,2}\mathcal{E}_{1,1}
&& (G_{4,4,4})_{0,6} = \frac{1}{6}\mathcal{E}_{0,2}\mathcal{E}_{0,2}\mathcal{E}_{0,2}.
\end{align*} 
Therefore, these are not ``new'' functions. As with the length two examples, each of the above functions has an associated Laplace-eigenvalue equation:
\begingroup
\allowdisplaybreaks
\begin{align*}
    &(\Delta+6)(G_{4,4,4})_{6,0}=- \mathbb{L}\mathbb{G}_4\mathcal{E}_{2,0}\mathcal{E}_{1,1}
    \\[1mm]
    &(\Delta+6)(G_{4,4,4})_{5,1}=-\mathbb{L}\mathbb{G}_4\mathcal{E}_{1,1}\mathcal{E}_{1,1} -2\mathbb{L}\mathbb{G}_4\mathcal{E}_{2,0}\mathcal{E}_{0,2}
    \\[1mm]
    &(\Delta+6)(G_{4,4,4})_{4,2}=- \mathbb{L}^2\mathbb{G}_4\overline{\mathbb{G}}_4\mathcal{E}_{2,0} -3 \mathbb{L}\mathbb{G}_4\mathcal{E}_{1,1}\mathcal{E}_{0,2}
    \\[1mm]
    &(\Delta+6)(G_{4,4,4})_{3,3}=- \mathbb{L}^2\mathbb{G}_4\overline{\mathbb{G}}_4\mathcal{E}_{1,1} -2 \mathbb{L}\mathbb{G}_4\mathcal{E}_{0,2}\mathcal{E}_{0,2} -2 \mathbb{L}\overline{\mathbb{G}}_4\mathcal{E}_{2,0}\mathcal{E}_{2,0}
    \\[1mm]    &(\Delta+6)(G_{4,4,4})_{2,4}=-\mathbb{L}^2\mathbb{G}_4\overline{\mathbb{G}}_4\mathcal{E}_{0,2} -3 \mathbb{L}\overline{\mathbb{G}}_4\mathcal{E}_{1,1}\mathcal{E}_{2,0}
    \\[1mm]
    &(\Delta+6)(G_{4,4,4})_{1,5}= - \mathbb{L}\overline{\mathbb{G}}_4\mathcal{E}_{1,1}\mathcal{E}_{1,1} -2\mathbb{L}\overline{\mathbb{G}}_4\mathcal{E}_{0,2}\mathcal{E}_{2,0}
    \\[1mm]
    &(\Delta+6)(G_{4,4,4})_{0,6}=-\mathbb{L}\overline{\mathbb{G}}_4\mathcal{E}_{0,2}\mathcal{E}_{1,1}.
\end{align*}
\endgroup
However, we are in search of Laplace-eigenvalue equations for ``new'' length three modular iterated integrals. We will get such equations in the next example.

\subsubsection{Example 2}\label{end}
We will now look at the case when $2a=4, 2b=2$ and $2c=2$. Theorem \ref{3MI3THM} gives:
\begingroup
\allowdisplaybreaks
\begin{align*}
&\partial (G_{6,4,4})_{8,0}=\mathbb{L}\mathbb{G}_6 (F^{(0)}_{4,4})_{4,0} = \frac{1}{2}\mathbb{L}\mathbb{G}_6\mathcal{E}_{2,0}\mathcal{E}_{2,0}
\\
&\partial (G_{6,4,4})_{7,1} - 8(G_{6,4,4})_{8,0}=\mathbb{L}\mathbb{G}_6 (F^{(0)}_{4,4})_{3,1} = \mathbb{L}\mathbb{G}_6\mathcal{E}_{2,0}\mathcal{E}_{1,1}
\\
&\partial (G_{6,4,4})_{6,2} - 7(G_{6,4,4})_{7,1}=\mathbb{L}\mathbb{G}_6 (F^{(0)}_{4,4})_{2,2} = \mathbb{L}\mathbb{G}_6 \mathcal{E}_{2,0}\mathcal{E}_{0,2} + \frac{1}{2}\mathbb{L}\mathbb{G}_6\mathcal{E}_{1,1}\mathcal{E}_{1,1}
\\
&\partial (G_{6,4,4})_{5,3} - 6(G_{6,4,4})_{6,2}=\mathbb{L}\mathbb{G}_6 (F^{(0)}_{4,4})_{1,3}=\mathbb{L}\mathbb{G}_6\mathcal{E}_{1,1}\mathcal{E}_{0,2}
\\
&\partial (G_{6,4,4})_{4,4} - 5(G_{6,4,4})_{5,3}=\mathbb{L}\mathbb{G}_6 (F^{(0)}_{4,4})_{0,4} = \frac{1}{2}\mathbb{LG}_6 \mathcal{E}_{0,2}\mathcal{E}_{0,2}
\\
& \partial (G_{6,4,4})_{3,5} - 4(G_{6,4,4})_{4,4}=0
\\[1mm]
& \partial (G_{6,4,4})_{2,6} - 3(G_{6,4,4})_{3,5}=0
\\[1mm]
& \partial (G_{6,4,4})_{1,7} - 2(G_{6,4,4})_{2,6}=0
\\[1mm]
& \partial (G_{6,4,4})_{0,8} - (G_{6,4,4})_{1,7}=0
\end{align*} 
and
\begin{align*}
& \overline{\partial} (G_{6,4,4})_{0,8} =
\mathbb{L}\overline{\mathbb{G}}_4(F^{(0)}_{6,4})_{0,6} + C_{6,4,4}\LL\overline{\mathbb{G}}_{10}
\\
& \overline{\partial} (G_{6,4,4})_{1,7} - 8(G_{6,4,4})_{0,8} = \mathbb{L}\overline{\mathbb{G}}_4(F^{(0)}_{6,4})_{1,5}
\\
& \overline{\partial}(G_{6,4,4})_{2,6} - 7(G_{6,4,4})_{1,7} = \mathbb{L}\overline{\mathbb{G}}_4(F^{(0)}_{6,4})_{2,4}
\\
& \overline{\partial}(G_{6,4,4})_{3,5} - 6(G_{6,4,4})_{2,6} = \mathbb{L}\overline{\mathbb{G}}_4(F^{(0)}_{6,4})_{3,3}
\\
& \overline{\partial}(G_{6,4,4})_{4,4} - 5(G_{6,4,4})_{3,5} = \mathbb{L}\overline{\mathbb{G}}_4(F^{(0)}_{6,4})_{4,2}
\\
& \overline{\partial}(G_{6,4,4})_{5,3} - 4(G_{6,4,4})_{4,4} = \mathbb{L}\overline{\mathbb{G}}_4(F^{(0)}_{6,4})_{5,1} \qquad
&& \overline{\partial}(G_{6,4,4})_{7,1} - 2(G_{6,4,4})_{6,2} = 0
\\
& \overline{\partial}(G_{6,4,4})_{6,2} - 3(G_{6,4,4})_{5,3} = \mathbb{L}\overline{\mathbb{G}}_4(F^{(0)}_{6,4})_{6,0}
&& \overline{\partial}(G_{6,4,4})_{8,0} - (G_{6,4,4})_{7,1} = 0,
\end{align*} 
\endgroup
where $C_{6,4,4}$ is some constant. The functions $(F^{(0)}_{6,4})_{r,s}$ cannot be expressed in terms of real analytic Eisenstein series, and therefore, we are not able to simplify the $\bar{\partial}$ equations like the $\partial$ equations. As before, each function has an associated Laplace-eigenvalue equation:
\begingroup
\allowdisplaybreaks
\begin{align*}
    &(\Delta+8)(G_{6,4,4})_{8,0}=- \mathbb{L}\mathbb{G}_6\mathcal{E}_{2,0}\mathcal{E}_{1,1}
    \\[0.5mm]
    &(\Delta+8)(G_{6,4,4})_{7,1}=-\mathbb{L}\mathbb{G}_6\mathcal{E}_{1,1}\mathcal{E}_{1,1} -2\mathbb{L}\mathbb{G}_6\mathcal{E}_{2,0}\mathcal{E}_{0,2}
    \\[0.5mm]
    &(\Delta+8)(G_{6,4,4})_{6,2}=-\mathbb{L}^2\mathbb{G}_6\overline{\mathbb{G}}_4\mathcal{E}_{2,0} -3\mathbb{L}\mathbb{G}_6\mathcal{E}_{1,1}\mathcal{E}_{0,2}
    \\[0.5mm]
    &(\Delta+8)(G_{6,4,4})_{5,3}= -\mathbb{L}^2\mathbb{G}_6\overline{\mathbb{G}}_4\mathcal{E}_{1,1}
    -6\mathbb{L}\overline{\mathbb{G}}_4(F^{(0)}_{6,4})_{6,0} - 2\mathbb{L}\mathbb{G}_6\mathcal{E}_{0,2}\mathcal{E}_{0,2}
    \\[0.5mm]    
    &(\Delta+8)(G_{6,4,4})_{4,4}=
    -\mathbb{L}^2\mathbb{G}_6\overline{\mathbb{G}}_4\mathcal{E}_{0,2}
    -5\mathbb{L}\overline{\mathbb{G}}_4(F^{(0)}_{6,4})_{5,1}
    \\[0.5mm]
    &(\Delta+8)(G_{6,4,4})_{3,5}=-4\mathbb{L}\overline{\mathbb{G}}_4(F^{(0)}_{6,4})_{4,2}
    \\[0.5mm]
    &(\Delta+8)(G_{6,4,4})_{2,6}=-3\mathbb{L}\overline{\mathbb{G}}_4(F^{(0)}_{6,4})_{3,3}
    \\[0.5mm]
    &(\Delta+8)(G_{6,4,4})_{1,7}=-2\mathbb{L}\overline{\mathbb{G}}_4(F^{(0)}_{6,4})_{2,4}
    \\[0.5mm]
    &(\Delta+8)(G_{6,4,4})_{0,8}=-\mathbb{L}\overline{\mathbb{G}}_4(F^{(0)}_{6,4})_{1,5}.
\end{align*}
\endgroup
However, even when $C_{6,4,4}$ is set to zero,
the $\partial, \bar{\partial}$ equations cannot be solved in terms of real analytic Eisenstein series. These functions are all ``new''. Therefore, we have found Laplace-eigenvalue equations for ``new'' length three modular iterated integrals, as desired. There are, of course, more values for $a$, $b$ and $c$ to try. In fact, we will get ``new'' functions whenever $a \neq b$ or $a \neq c$.

\subsubsection{Example 3}\label{endex3}

We finish by looking at a case where $\bar{g}^{(0)}_{2b+2,2c+2}$ is not known to vanish.
We let $2b=4$ and $2a=2c=2$, then Theorem \ref{3MI3THM} gives:
\begingroup
\allowdisplaybreaks
\begin{align*}
    &\partial (G_{4,6,4})_{8,0}= \mathbb{L}\mathbb{G}_4 (F^{(0)}_{6,4})_{6,0}
\\
&\partial (G_{4,6,4})_{7,1} - 8(G_{4,6,4})_{8,0}
 = \mathbb{L}\mathbb{G}_4 (F^{(0)}_{6,4})_{5,1} 
\\
&\partial (G_{4,6,4})_{6,2} - 7(G_{4,6,4})_{7,1}
= \mathbb{L}\mathbb{G}_4 (F^{(0)}_{6,4})_{4,2} 
\\
&\partial (G_{4,6,4})_{5,3} - 6(G_{4,6,4})_{6,2}
= \mathbb{L}\mathbb{G}_4 (F^{(0)}_{6,4})_{3,3} 
\\
&\partial (G_{4,6,4})_{4,4} - 5(G_{4,6,4})_{5,3}
=  \mathbb{L}\mathbb{G}_4 (F^{(0)}_{6,4})_{2,4}
\\
&\partial (G_{4,6,4})_{3,5} - 4(G_{4,6,4})_{4,4}
 = \mathbb{L}\mathbb{G}_4 (F^{(0)}_{6,4})_{1,5}
\\
&\partial (G_{4,6,4})_{2,6} - 3(G_{4,6,4})_{3,5}
=\mathbb{L}\mathbb{G}_4 (F^{(0)}_{6,4})_{0,6}
\\
&\partial (G_{4,6,4})_{1,7} - 2(G_{4,6,4})_{2,6}= 0
\\
&\partial (G_{4,6,4})_{0,8} - (G_{4,6,4})_{1,7}= 0
\end{align*}
\endgroup
and
\begin{align*}
& \overline{\partial}(G_{4,6,4})_{0,8} = \mathbb{L}\overline{\mathbb{G}}_4(F^{(0)}_{4,6})_{0,6} 
+ C^{(0)}_{6,4}\mathbb{L}\overline{\mathbb{G}}_8\mathcal{E}_{0,2} +  C_{4,6,4}\LL\overline{\mathbb{G}}_{10}
\\[0.5mm]
& \overline{\partial}(G_{4,6,4})_{1,7} - 8(G_{4,6,4})_{0,8} = \mathbb{L}\overline{\mathbb{G}}_4(F^{(0)}_{4,6})_{1,5} + C^{(0)}_{6,4}\mathbb{L}\overline{\mathbb{G}}_8\mathcal{E}_{1,1}
\\[0.5mm]
&\overline{\partial}(G_{4,6,4})_{2,6} - 7(G_{4,6,4})_{1,7}  = \mathbb{L}\overline{\mathbb{G}}_4(F^{(0)}_{4,6})_{2,4} + C^{(0)}_{6,4}\mathbb{L}\overline{\mathbb{G}}_8\mathcal{E}_{2,0}
\\[0.5mm]
& \overline{\partial}(G_{4,6,4})_{3,5} - 6(G_{4,6,4})_{2,6} = \mathbb{L}\overline{\mathbb{G}}_4(F^{(0)}_{4,6})_{3,3}
\\
&\overline{\partial}(G_{4,6,4})_{4,4} - 5(G_{4,6,4})_{3,5}  = \mathbb{L}\overline{\mathbb{G}}_4(F^{(0)}_{4,6})_{4,2}
\\
&\overline{\partial}(G_{4,6,4})_{5,3} - 4(G_{4,6,4})_{4,4} = \mathbb{L}\overline{\mathbb{G}}_4(F^{(0)}_{4,6})_{5,1}
\\
& \overline{\partial}(G_{4,6,4})_{6,2} - 3(G_{4,6,4})_{5,3}  = \mathbb{L}\overline{\mathbb{G}}_4(F^{(0)}_{4,6})_{6,0}
\\
& \overline{\partial}(G_{4,6,4})_{7,1} - 2(G_{4,6,4})_{6,2} = 0
\\
& \overline{\partial}(G_{4,6,4})_{8,0} - (G_{4,6,4})_{7,1} = 0,
\end{align*}
where $C^{(0)}_{6,4}$ and $C_{4,6,4}$ are some constants. 
Even when $C^{(0)}_{6,4}$ and $C_{4,6,4}$ are set to zero, these equations are not solvable in terms of real analytic Eisenstein series. They are ``new'' functions. The associated Laplace equations are as follows:
\begingroup
\allowdisplaybreaks
\begin{align*}
     &(\Delta+8)(G_{4,6,4})_{8,0}= -\LL\GG_4 (F^{(0)}_{6,4})_{5,1}
    \\
    &(\Delta+8)(G_{4,6,4})_{7,1}= -2\LL\GG_4 (F^{(0)}_{6,4})_{4,2}
    \\[0.5mm]
    &(\Delta+8)(G_{4,6,4})_{6,2}= -3\LL\GG_4 (F^{(0)}_{6,4})_{3,3} - \LL^2\mathbb{G}_4\overline{\mathbb{G}}_4\mathcal{E}_{4,0}
    \\[0.5mm]
    &(\Delta+8)(G_{4,6,4})_{5,3}= -4\LL\GG_4 (F^{(0)}_{6,4})_{2,4} - \LL^2\mathbb{G}_4\overline{\mathbb{G}}_4\mathcal{E}_{3,1} - 6\LL\overline{\mathbb{G}}_4(F^{(0)}_{4,6})_{6,0}
    \\[0.5mm]
    &(\Delta+8)(G_{4,6,4})_{4,4}= -5\LL\GG_4 (F^{(0)}_{6,4})_{1,5} - \LL^2\mathbb{G}_4\overline{\mathbb{G}}_4\mathcal{E}_{2,2} - 5\LL\overline{\mathbb{G}}_4(F^{(0)}_{4,6})_{5,1}
    \\[0.5mm]
    &(\Delta+8)(G_{4,6,4})_{3,5}= -6\LL\GG_4 (F^{(0)}_{6,4})_{0,6} - \LL^2\mathbb{G}_4\overline{\mathbb{G}}_4\mathcal{E}_{1,3} - 4\LL\overline{\mathbb{G}}_4(F^{(0)}_{4,6})_{4,2}
    \\[0.5mm]
    &(\Delta+8)(G_{4,6,4})_{2,6}= - \LL^2\mathbb{G}_4\overline{\mathbb{G}}_4\mathcal{E}_{0,4} - 3\LL\overline{\mathbb{G}}_4(F^{(0)}_{4,6})_{3,3} - C^{(0)}_{6,4}\LL^2 \overline{\mathbb{G}}_8\mathbb{G}_4
    \\[0.5mm]
    &(\Delta+8)(G_{4,6,4})_{1,7}= -2\LL\overline{\mathbb{G}}_4(F^{(0)}_{4,6})_{2,4} - 2C^{(0)}_{6,4}\LL \overline{\mathbb{G}}_8\mathcal{E}_{2,0}
    \\[0.5mm]
    &(\Delta+8)(G_{4,6,4})_{0,8}= -\LL\overline{\mathbb{G}}_4(F^{(0)}_{4,6})_{1,5} - C^{(0)}_{6,4}\LL \overline{\mathbb{G}}_8\mathcal{E}_{1,1}.
\end{align*}
\endgroup
We note that the summands involving the constant term $C^{(0)}_{6,4}$ are simply multiples of the Laplace equations given by \eqref{ex3a} -- \eqref{ex3c}. Therefore, the above Laplace equations can easily be edited to remove these terms (if this became necessary).

Of course, there is great appeal in being able to determine the value of the constants $C^{(k)}_{2b+2,2c+2}$ and $C_{2a+2,b+2,2c+2}$ for specific $a$, $b$, $c$ and $k$. It would be especially useful to be able to determine when these constants disappear.

As mentioned in Section \ref{research}, it would not be surprising if the length three Laplace equations were to be used in expressing modular graph functions of three-loops or higher, rather than the two-loop functions. This possibility is hinted at in the Laplace equation for the three-loop modular graph function $C_{1,1,2,2}$, appearing in equations (4.47) and (4.48) of \cite{ellipzeta} and in equation (B.77) of \cite{Cabcd}. This Laplace equation contains a term that is cubic in a real analytic Eisenstein series, as well as a term involving a two-loop modular graph function multiplied by an Eisenstein series:
\begin{align*}
    &(\Delta+12)\left(C_{1,1,2,2}-\dfrac{8}{45}\mathbb{L}^4\mathcal{E}_{1,1}\mathcal{E}_{3,3} - \dfrac{8}{9}\mathbb{L}^4\mathcal{E}_{2,2}\mathcal{E}_{2,2}\right) \\ & \quad
    = \dfrac{128}{3}\mathbb{L}^3\mathcal{E}_{1,1}\mathcal{E}_{1,1}\mathcal{E}_{1,1} +  \dfrac{64}{9}\mathbb{L}^4\mathcal{E}_{2,2}\mathcal{E}_{2,2} 
    + \dfrac{16}{15}\mathbb{L}^4\mathcal{E}_{1,1}\mathcal{E}_{3,3} - \dfrac{4}{4725}\mathbb{L}^5\mathcal{E}_{5,5} \\& \quad \quad+ 56C_{1,2,3} + 12C_{1,1,4} + \dfrac{22}{3}C_{2,2,2} - 16\mathbb{L}\mathcal{E}_{1,1}C_{1,1,2} + 4 P_{1,1;1,2;1} - 4P_{1,1;1,1;2},
\end{align*}
where $P_{1,1;1,2;1}$ and $P_{1,1;1,1;2}$ are three-loop modular graph functions.

This Laplace equation also demonstrates the appeal of extending Definition \ref{defabc} to $k>0$, to generate ``new'' functions $(G^{(k)}_{2a+2,2b+2,2c+2})_{r,s}$. For example, if the term $\mathbb{L}^3\mathcal{E}_{1,1}\mathcal{E}_{1,1}\mathcal{E}_{1,1}$ were to match up with any of the functions $(G^{(k)}_{2a+2,2b+2,2c+2})_{r,s}$, such a function would be of the form $(G^{(k)}_{4,4,4})_{r,s}$. However, we have seen that restricting to $k=0$ does not lead to any ``new'' functions for $2a=2b=2c=2$. Therefore, if $\mathbb{L}^3\mathcal{E}_{1,1}\mathcal{E}_{1,1}\mathcal{E}_{1,1}$ were to match up with any ``new'' functions, $k$ would have to be greater than $0$.

\bibliography{paperrefs} 
\bibliographystyle{ieeetr}

\end{document}